\documentclass[oneside,11pt]{article}

\usepackage{geometry}
 \usepackage[left]{lineno}

\usepackage{graphicx,epstopdf}

\usepackage{amsopn}
\usepackage{bold-extra}
\usepackage{mathtools}
\usepackage{psfrag}
\usepackage{amssymb}
\usepackage{amsmath}
\usepackage{graphicx}
\usepackage{xcolor}
\usepackage{amssymb}
\usepackage{units}

\usepackage{xspace}
\usepackage{bold-extra}

\usepackage{wasysym}
\usepackage[color=green!40]{todonotes}

\usepackage{soul}
\usepackage{xcolor}
\colorlet{lred}{red!40}
\colorlet{lblue}{blue!40}
\colorlet{lgreen}{green!40}

\definecolor{mixc}{cmyk}{0.5,0.5,0.5,0}
\colorlet{mixl}{mixc!30}

\usepackage{enumitem}

\usepackage{enumitem}
\setitemize{label=\scriptsize{$\blacksquare$}}
\setenumerate{label=(\alph*)}
\usepackage{amssymb,amsthm}

\newtheorem{theorem}{Theorem}
\newtheorem{definition}[theorem]{Definition}
\newtheorem{corollary}[theorem]{Corollary}
\newtheorem{lemma}[theorem]{Lemma}
\newtheorem{remark}[theorem]{Remark}

\newtheorem{example}[theorem]{Example}

\setenumerate{label=(\alph*)}

\numberwithin{equation}{section}
\numberwithin{figure}{section}
\numberwithin{theorem}{section}

\newcommand{\imi}{\mathsf{i}}

\newcommand{\source}{h}
\newcommand{\hilbert}[1]{\mathcal{H}}

\newcommand{\C}{\mathbb C}
\newcommand{\R}{\mathbb R}

\newcommand{\N}{\mathbb N}
\newcommand{\sph}{\mathbb S}
\newcommand{\edot}{\,\cdot\,}

\newcommand{\Wo}{\mathbf W}

\newcommand{\Mo}{\mathbf M}
\newcommand{\Po}{\mathbf P}

\newcommand{\Ao}{\mathbf A}

\newcommand{\Do}{\mathbf D}

\newcommand{\RR}{R}
\newcommand{\ph}{\varphi}

\renewcommand{\Re}{\operatorname{Re}}
\renewcommand{\Im}{\operatorname{Im}}
\DeclareMathOperator*{\real}{Re}
\DeclareMathOperator*{\imag}{Im}

\newcommand{\rmd}{\mathrm d}
\newcommand{\ds}{\mathrm{d}S}
\newcommand{\eps}{\epsilon}
\newcommand{\coloneqq}{:=}

\newcommand{\tmin}{{t_0}}
\newcommand{\tmax}{{T}}
\newcommand{\om}{\omega}

\newcommand{\Omin}{\Omega_0}
\newcommand{\Omout}{\Omega}
\newcommand{\al}{\alpha}
\newcommand{\la}{\lambda}

\newcommand{\Sr}{\mathcal S}
\newcommand{\diam}{\operatorname{diam}}
\newcommand{\dist}{\operatorname{dist}}

\newcommand\abs[1]{\left\vert#1\right\vert}
\newcommand\sabs[1]{\lvert#1\rvert}
\newcommand\norm[1]{\left\Vert#1\right\Vert}
\newcommand\snorm[1]{\Vert#1\Vert}
\newcommand{\enorm}{\left\|\;\cdot\;\right\|}
\newcommand{\kl}[1]{\left(#1\right)}
\newcommand{\ekl}[1]{\left[#1\right]}

\newcommand\inner[2]{\left\langle#1,#2\right\rangle}

\newcommand\rest[2]{{#1}\vert_{#2}}
\newcommand\sfrac[2]{{#1}/{#2}}

\newcommand{\hnum}{{\tt h}}
\newcommand{\gnum}{{\tt g}}
\newcommand{\mnum}{{\tt m}}
\newcommand{\Wnum}{\boldsymbol{\tt{W}}}
\newcommand{\Pnum}{\boldsymbol{\tt{P}}}
\newcommand{\Cnum}{\tt{C}}
\newcommand{\Bnum}{\boldsymbol{\tt{B}}}
\newcommand{\Mnum}{\boldsymbol{\tt{M}}}

   \DeclarePairedDelimiterX\set[1]\{\}{%
      
      #1
   }

\newcommand{\sr}{\mathcal{S}}

\newcommand{\pd}{\partial}

 \newcommand{\green}{G}
\newcommand{\fourier}{\mathcal F}
\newcommand{\ifourier}{\mathcal{F}^{-1}}

 \newcommand{\x}{\mathbf{x}}
 \newcommand{\vx}{\mathbf{x}}

\usepackage{hyperref}

\allowdisplaybreaks

\title{Iterative Methods for Photoacoustic Tomography\\in Attenuating Acoustic Media}

\author{Markus Haltmeier\footnotemark[1]  \and Richard Kowar\footnotemark[1]  \and Linh V. Nguyen\footnotemark[2]}

\date{\small \footnotemark[1] Department of Mathematics, University of Innsbruck\\
Technikerstrasse 13, A-6020 Innsbruck, Austria\\
  { \tt \{Markus.Haltmeier,Richard.Kowar\}@uibk.ac.at}\\[1em]
\small \footnotemark[2]
Department of Mathematics, University of Idaho\\
875 Perimeter Dr, Moscow, ID 83844, {\tt lnguyen@uidaho.edu}}

\begin{document}
\maketitle

%\todo[color=red!50,inline]{NOTE: This is paper draft version January 24, 2017}

\begin{abstract}
The development of efficient and accurate reconstruction methods is an important aspect of
tomographic imaging. In this article, we address this issue for photoacoustic tomography.  To this aim, we
use models for acoustic wave propagation accounting for frequency dependent attenuation according to a wide class of attenuation laws that may include memory.
We formulate the inverse problem of photoacoustic tomography in attenuating medium as an ill-posed operator equation in a Hilbert space framework that is tackled by iterative regularization methods.
Our approach comes with a clear convergence analysis.
For that purpose we derive explicit expressions for the adjoint problem that can efficiently be implemented.
In contrast to time reversal, the employed adjoint wave equation is again damping and, thus has a stable solution. This stability property can be clearly seen in our numerical results.
Moreover, the presented numerical results  clearly demonstrate the efficiency
and accuracy of the derived iterative reconstruction algorithms in various situations
including the limited view case.


\end{abstract}

\linenumbers
\nolinenumbers

\section{Introduction}
\label{sec:intro}

Photoacoustic tomography (PAT) is an emerging coupled-physics  imaging modality that combines the high spatial resolution of ultrasound imaging with the high contrast of optical imaging (the basic principles are illustrated in Figure~\ref{fig:pat}).
Potential medical applications include imaging of tumors, visualization of vasculature or scanning of melanoma  \cite{beard2011biomedical,KruKisReiKruMil03,ntziachristos2005looking,wang2012photoacoustic}.
In this article we consider PAT using the following general model for acoustic wave propagation in attenuating media,
\begin{linenomath} \begin{equation}  \label{eq:wavealpha}
	\kl{ \Do_\al  +
	\frac{1}{c_0}\frac{\partial}{\partial t }  }^2  p_\al(x,t) -
    \Delta p_\al(x,t)    =  \delta'(t) h(x)
	\quad \text{ for }(x,t) \in \R^d \times \R \,.
\end{equation}\end{linenomath}
Here   $h \colon \R^d \to \R$ is the photoacoustic (PA) source,
$c_0 > 0$ is a constant, and $\Do_\al$
is the time convolution operator  associated with the inverse Fourier transform of the
complex valued attenuation function $\al \colon \R \to \C$.
Dissipative pressure wave equation models that can be cast in the form~{\eqref{eq:wavealpha}} can be found
in~\cite{elbau2016singular,hanyga2014dispersion,kinsler1999fundamentals,kowar2010integral,kowar2012photoacoustic,kowar2011causality,lariviere2006image,sushilov2004frequency,szabo1994time}. The particular form of $\al$  depends on the used acoustic attenuation model and various different models have  been proposed for PAT (see \cite{kowar2012photoacoustic} for an overview).

The inverse problem of PAT consists in recovering the source term $h$ from  observations of  $p_\al$ on an observation surface $\Gamma \subseteq \R^d$ outside its support and for times
$t \in (0,\tmax)$.  Taking attenuation into account is essential for  high resolution PAT since ignoring attenuation may significantly blur
the reconstructed image.

\begin{figure}[tbh!]
\centering
  \includegraphics[width=\textwidth]{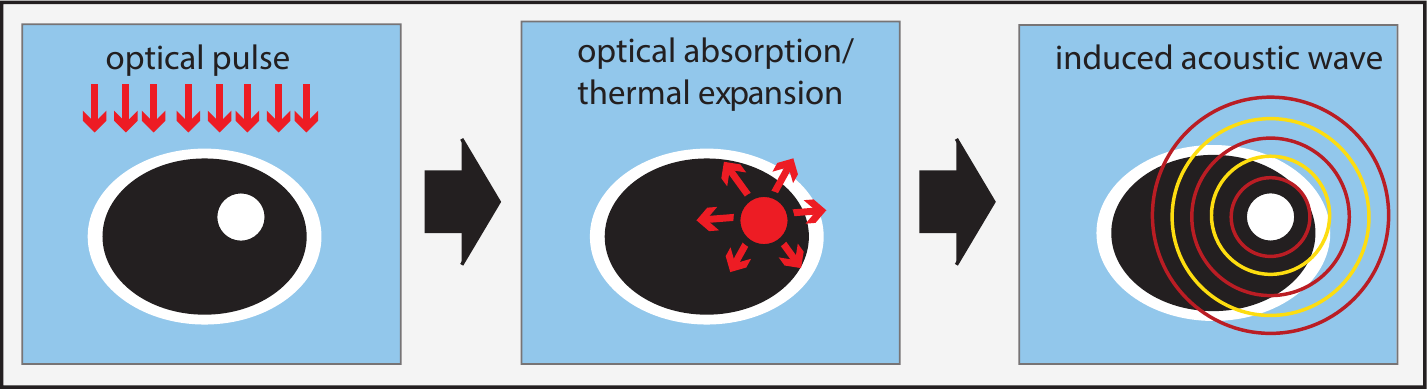}
\caption{\textsc{Basic principles of PAT.}\label{fig:pat}  A semitransparent sample is illuminated with a  short optical pulse. Due to optical absorption and subsequent thermal expansion an acoustic pressure wave is induced within  the sample. The pressure waves are measured outside of the sample and used to reconstruct an image of the interior.}
\end{figure}

\subsection{Our approach}

The inverse problem of PAT can be formulated as the
problem of estimating  $h$ from approximate data
$g^\delta \simeq \Wo_\alpha h$, where $\Wo_\alpha$ maps the PA source $h$ to the solution of \eqref{eq:wavealpha} restricted to $\Gamma \times (0,T)$. In this paper, we propose the use of regularization methods for stably inverting  the operator $\Wo_\al $. In particular, we apply the Landweber  method, which is a well established regularization method.
A main ingredient in the Landweber method is the numerical evaluation of the adjoint
$\Wo_\al^*$. For that purpose, we derive two explicit expressions for  the adjoint.
The first one takes the form  of an explicit formula for the  adjoint operator  and
will be used in our numerical implementation.  The second one involves the solution of an adjoint attenuated wave equation.
 We emphasize that our inversion approach is universal, in the sense, that it can be applied to a wide range of different attenuation models, a general measurement geometry as well as limited data problems.

\subsection{Comparison to previous and related work}

In the case of  vanishing attenuation  $\al=0$, the attenuated wave
equation~\eqref{eq:wavealpha} reduces to the standard  wave equation
\begin{linenomath} \begin{equation}  \label{eq:wave0}
	\frac{1}{c_0^2}\frac{\partial^2}{\partial t^2 }   p_0(x,t) -
	\Delta p_0(x,t)    =  \delta'(t) h(x)
	\quad \text{ for }(x,t) \in \R^d \times \R
\end{equation}\end{linenomath}
with sound speed $c_0$. Recovering  the source term $h$ in \eqref{eq:wave0} from boundary data is the standard  problem in PAT
and  various methods for its solution have been derived in the
recent years.  These approaches can be classified
in direct methods, time reversal and iterative  approaches. Direct methods are based on explicit solutions for the inverse problem that have been derived in the Fourier domain \cite{agranovsky2007uniqueness,haltmeier2007thermoacoustic,kunyansky2007series,xu2002exact}
as well as in the spatial domain~\cite{finch2007inversion,finch2004determining,haltmeier13inversion,haltmeier2014universal,haltmeier2015universal,
kunyansky2007explicit,kunyansky2015inversion,natterer2012photo,nguyen2009family,palamodov2012uniform,salman2014inversion,xu2005universal}.
In the time reversal technique, the wave equation \eqref{eq:wave0} is solved backwards in time where the
measured data are used as boundary values in the time reversed wave equation \cite{burgholzer2007exact,finch2004determining,Hristova2008,nguyen2016dissipative,treeby2010kwave}.
Discrete iterative approaches, on the other hand, are usually based on a discretization of the forward problem
together with numerical solution methods for solving the resulting system of linear equations
\cite{deanben2012accurate,paltauf2007experimental,paltauf2002iterative,rosenthal2013current,zhang2009effects,wang2014discrete,wang2012investigation}. Recently, iterative schemes in a Hilbert space settings have also  been introduced and studied; see \cite{arridge2016adjoint,belhachmi2016direct,haltmeier2016iterative}.
In this paper we generalize the iterative Hilbert space approach to attenuating
media.

The case of  non-vanishing attenuation is much less investigated and
existing methods are very different from our approach. One
class of reconstruction  methods  uses  the following two-stage procedure:
In a first step, by solving an ill-posed integral equation the (idealized) un-attenuated pressure data $p_0(z, \edot)$ are  estimated from the attenuated data $p_\al(z, \edot)$. In the second step, the standard  PAT problem is solved. Such a two step method  has been proposed and implemented for the power law in~\cite{lariviere2005image,lariviere2006image},
and later been used in~\cite{ammari2012photoacoustic,kowar2012photoacoustic}
for various attenuation laws. Compared to two stage approaches, in the  single step approach it is easier to include prior information available  in the image domain, such
as positivity  of the PAT source (compare Section~{\ref{sec:landweber}}).
Furthermore, in the limited data case, where the measurement  surface does not fully enclose the PA source,
 the  second step in the two-stage approach is again a
non-standard problem for PAT, for which iterative  methods
can be applied. In such a situation  it seems reasonable to directly
apply iterative methods  to the attenuated data, as considered
in the present paper.

A different class of algorithms extends the time reversal technique to the
attenuated case (see
\cite{acosta2017thermoacoustic,ammari2011time,burgholzer2007compensation,homan2013multi,kalimeris2013photoacoustic,
kowar2014time,palacios2016reconstruction,treeby2010photoacoustic}).
Note that the time reversal of the attenuated wave equation yields a noise amplifying
equation. Therefore regularization methods have to be incorporated in
its numerical solution.
Opposed to the time reversal, the adjoint wave equation used in our approach is again damping and no regularization is required for its stable solution.
This yields a clear convergence analysis for our method by using standard
regularization theory~\cite{engl1996regularization,kaltenbacher2008iterative,scherzer2009variational}.
We are not aware of similar existing results for  PAT in attenuating acoustic media.
 The approaches which are closest to our work seem~\cite{huang2013full,javaherian2017multi}.
In \cite{huang2013full} discrete iterative methods are considered,
where  the problem is first discretized and the adjoint is computed from
the discretized problem. Further, both works \cite{huang2013full,javaherian2017multi} use attenuation models based on the
fractional Laplacian  (see \cite{chen2004fractional,treeby2010modeling}) which yields an equation that is non-local in space.
It is not obvious how to extend these approaches to model~\eqref{eq:wavealpha} which can also include memory.

\subsection{Notation}

For $k \in \N$, we  write  $\sr(\R^{k+1})$ for the  Schwartz space  of rapidly decreasing functions $f \colon \R^{k+1}  \to \C$, and  $\sr'(\R^{k+1})$ for its dual, the space of tempered distributions.
Further we write $\fourier_t$  for the Fourier transform
in the temporal variable, defined by $(\fourier_t f)(x,\om) = \int_{\R} f(t) e^{\imi \om t} f(x,t)\rmd t$
for  $f\in \sr(\R^{k+1})$ and extended
by duality to tempered distributions. A tempered distribution
 in $\sr'(\R^{k+1})$ will be called causal (in the last component)
 if  it vanishes  for $t<0$. Finally, for $\al \in \sr(\R)$ we denote
 by $\Do_\al$ the time convolution operator with
 kernel $\ifourier_t(\al)$.

\subsection{Outline}

In Section~\ref{sec:pat} we formulate the forward operator
of the PAT in attenuating acoustic media in a Hilbert space framework.
We show that it is  continuous between $L^2$-spaces and
give an explicit expression for its solution.  We further derive
two expressions for the adjoint operator. In
Section~\ref{sec:inverse} we solve the corresponding inverse problem using the
Landweber regularization, present convergence results, and
give details for its actual implementation. Numerical results  are presented in Section~\ref{sec:num}, and a conclusion  is given in
Section~\ref{sec:conclusion}. Finally, in the appendix we present details for the wave equation formulation of the
adjoint operator.

\section{PAT in attenuating acoustic media}
\label{sec:pat}

Throughout this paper we assume that  $\alpha \colon \R\to\C$
is a weakly causal attenuation function, defined as follows.

\begin{definition}[Weakly causal attenuation function]
A\label{def:A} function $\alpha \colon \R\to\C$ is called
weakly causal attenuation function, if the following assertions
hold true:
\begin{enumerate}[label=(A\arabic*)]
\item\label{def:A1} $\real(\al)$ is even and $\imag(\al)$ is odd;
\item\label{def:A2} $\om \mapsto\real(\al(\om))$ is monotonically increasing for positive $\om$;
\item\label{def:A3} $\ifourier_t(\al)(t)$ vanishes for $t<0$.
\end{enumerate}
\end{definition}

Note that~\ref{def:A1}  implies that the inverse Fourier transforms of $\alpha$ and  $e^{-\alpha(\omega)\,\sabs{\x}}$ are real valued. The second condition reflects increasing attenuation with increasing  frequency. It is not essential and may be replaced by a similar property.
The condition \ref{def:A3} implies causality of the Greens function $\green_\alpha$ (i.e. $\green_\alpha (\edot,t) = 0 $ for $t<0$) and further
is equivalent to the Kramers-Kr{\"o}nig relations  (see {\eqref{eq:kk1}}, {\eqref{eq:kk2}} below). Examples for weakly causal  attenuation functions are given in
Subsection~\ref{sec:examples}.

\subsection{Attenuated wave equations}

We describe  acoustic waves in attenuation  media by the
integro-differential equation
\begin{linenomath} \begin{equation}  \label{eq:w}
\left\{     \begin{aligned}
    & \kl{ \Do_\al  +
	\frac{1}{c_0}
	\frac{\partial}{\partial t }  }^2  p_\al(\x,t)
	- \Delta p_\al(\x,t)
	=  s(\x,t)
	&& \text{ for } (\x,t) \in \R^{d+1} \,,
\\
&\quad p_\al(\edot,t) = 0 && \text{ for } t < 0 \,.
\end{aligned}
\right. \end{equation}\end{linenomath}
Here $s$ is  a source term and $\alpha \colon \R\to\C$ a weakly causal attenuation function. For any causal $s \in \sr'(\R^{d+1})$,  the attenuated wave equation \eqref{eq:w} has a causal solution $p_\al \in \sr'(\R^{d+1})$.
In particular, this implies the existence of
causal Greens function that takes the form (see~\cite{kowar2012photoacoustic})
\begin{linenomath} \begin{equation} \label{eq:green}
\green_\alpha (\vx,t)
= \frac{K_\al\left(\vx,t-\frac{\abs{\vx}}{c_0}\right)}{4\pi \abs{\vx}}
\quad \text{ with } \quad
K_\al(\x , t )
\coloneqq
\ifourier_t \kl{ e^{-\abs{\x} \alpha} }(t)\,.
\end{equation}\end{linenomath}
The Greens function represents  a spherical wave in attenuating acoustic
media  that originates at location $\x=0$ and time $t=0$. It
satisfies~\eqref{eq:wavealpha} with right hand side  $s(\x,t)=-\delta(\x)\delta(t)$.
If $K_\al(\x, \edot - \tfrac{\abs{x}}{c_0})$ is causal for every $\x$,  then  the dissipative  Green function $\green_\al$, defined in~\eqref{eq:green},  has a finite wave front speed $\leq c_0$.
Attenuation laws with finite wave front speed are
called strongly causal in~\cite{kowar2012photoacoustic}.

Throughout we refer to the convolution of the source $s$ with $\green_\alpha$  as
\emph{the causal} solution of~\eqref{eq:w}. For the model of~\cite{kowar2011causality} uniqueness is
shown in  \cite{kowar2010integral}.  Note that the  proof of \cite{kowar2010integral} can be generalized  to any
weakly causal attenuation law. This implies uniqueness of a solution of \eqref{eq:w}.

\begin{definition}
Let $\al \colon \R \to \C$ be a weakly causal attenuation function.
Then, for any  $r \in \R$ we
define $m_\al (\edot, r) \in \Sr'(\R)$ by
\begin{linenomath} \begin{equation}  \label{eq:Malpha}
\forall \om \in \R \colon \quad \fourier_t (m_\al  (\edot,r)  )(\om) \coloneqq    \frac{\om}{ \om/c_0 + \imi \al(\om)}
\, e^{\imi (\om / c_0  + \imi \al(\om)) \abs{r}}  \,.
\end{equation}\end{linenomath}
\end{definition}

The following result derived in~\cite{kowar2010integral,kowar2012photoacoustic}
will be frequently used in this paper.

\begin{lemma}[Relation between attenuated and un-attenuated pressure]
\label{lem:Malpha}
Let $\al$ be a non-vanishing weakly causal attenuation function.
Then $m_\al$ is $C^\infty$ on $\R^2\setminus \set{(0,0)}$.
Moreover,
 \begin{linenomath} \begin{equation}  \label{eq:Mpalpha}
\forall (x,t) \in \R^d \times (0, \infty) \colon \quad
p_\al(\x ,t)    =
\int_0^t  m_\al(t,r) p_0(\x,r) \rmd r \,,
\end{equation}\end{linenomath}
where $p_\al$ and $p_0$ denote the causal solutions
of \eqref{eq:wavealpha} and \eqref{eq:wave0}, respectively.
\end{lemma}

\begin{proof}
See \cite[Theorem~1 and Lemma~1]{kowar2010integral}.
\end{proof}

\subsection{Examples for causal attenuation laws}
\label{sec:examples}

In this subsection, we give  particular examples for
causal attenuation laws that we use in this paper:  the power law
(see \cite{szabo1955causal,szabo1994time,waters2000timedomain}),
the model of Kowar, Scherzer, Bonnefond (see~\cite{kowar2012photoacoustic,kowar2011causality})
and the  model of Nachman, Smith and Waag (see~\cite{nachman1990equation}).

\begin{remark}[Kramers-Kronig relations]
A central property that should be satisfied by~\eqref{eq:wavealpha}
is causality of the corresponding Greens function $G_\al$.
Causality of $G_\al$ is equivalent to assumption \ref{def:A3}, the
causality of $\ifourier_t (\alpha)$ (see \cite{kowar2012photoacoustic}).
Widely use criteria for the causality of $\ifourier_t (\al)$
are the  Kramers-Kronig relations (see~\cite{kramers1927diffusion,kronig1926theory})
\begin{linenomath}\begin{align}        \label{eq:kk1}
	\Re[ \al(\om) ] &=
	\phantom{-} \frac{1}{\pi} \, \mathrm{P.V.}\! \int_\R
	\frac{\Im[ \al(\om') ]}{\om'-\om} d\om' \,,\\
        \label{eq:kk2}
	\Im[ \al(\om) ] &=
	- \frac{1}{\pi} \, \mathrm{P.V.}\!\int_\R
	\frac{\Re[ \al(\om') ]}{\om'-\om}d\om' \,.	
\end{align}\end{linenomath} In fact, according to Titchmarsh's theorem \cite[Theorem~95]{titchmarsh1986introduction}
for square integrable $\al$,  the causality of $\ifourier_t (\al)$ is equivalent to \eqref{eq:kk1} as well as to~\eqref{eq:kk2}. In such a situation, if the
imaginary part of the weakly causal attenuation function is known, then its real part
is uniquely determined by~\eqref{eq:kk1}.

Typical acoustic attenuation laws, however, are not square integrable
(see the examples below). In this case, the Kramers-Kronig
relations cannot be applied directly. Nevertheless, the method of
subtractions  allows extension to attenuation functions with $\al (\om) = \mathcal{O}(\omega^n)$ as $\om \to \infty$ with $n\in \N$
(see~\cite[Section~1.7]{nussenzveig1972causality}). In such a situation,
given the imaginary part $\Im[ \al]$, the Kramers-Kronig relation \eqref{eq:kk1}
determines the real part $\Re[ \al ]$ up to $n+1 $ additive constants.
As a concrete example, consider the case
where  $\al (\om) = \mathcal{O}(\omega)$. Then the Kramers-Kronig
relations yield
%\begin{linenomath}\begin{align}%        \label{eq:kk1a}
%	\Re[ \al(\om) ] &=
%    \Re[a]
%	+ \frac{1}{\pi} \, \mathrm{P.V.} \int_\R
%	\frac{\Im[ \al(\om') ]- \Im[a]}{\om'-\om} d\om' \,,\\
%        \label{eq:kk2a}
%	\Im[ \al(\om) ]
%     &=
%	\Im[a] - \frac{1}{\pi} \, \mathrm{P.V.} \int_\R
%	\frac{\Re[\al(\om')] - \Re[a]}{\om'-\om}d\om' \,.	
%\end{align}\end{linenomath}
\begin{linenomath}\begin{align}        \label{eq:kk1a}
	\Re[ \al(\om) ] &=
    \Re[ \al(\om_0) ]
	+ \frac{\om-\om_0}{\pi} \; \mathrm{P.V.}\! \int_\R
	\frac{\Im[ \al(\om') - \al(\om_0)]}{\om'-\om_0}
\, \frac{d\om'}{\om-\om'} \,,\\
        \label{eq:kk2a}
	\Im[ \al(\om) ]
     &=
	\Im[ \al(\om_0) ] - \frac{\om-\om_0}{\pi} \;
\mathrm{P.V.}\!\int_\R
	\frac{\Re[\al(\om') - \al(\om_0)]}{\om'-\om_0}
\, \frac{d\om'}{\om-\om'} \,.	
\end{align}\end{linenomath} In particular,  the imaginary part of the attenuation function
determines its real part provided that $\Im[ \al(\om_0) ]$, for some fixed $\om_0$,
is given.
For the general case $\al (\om) = \mathcal{O}(\omega^n)$
see~\cite{nussenzveig1972causality}.
\end{remark}

In the following $\alpha_0$, $c_0$, $c_\infty$, $\tau_1$
and $\gamma$ denote positive constants.

\begin{example}[Power law]
In the  power law model,  the complex attenuation function takes the form
\begin{linenomath} \begin{equation} \label{eq:powlaw}
   \alpha \colon \R \to \C \colon
   \omega
   \mapsto
   a_0\,(-\imi \,\omega)^\gamma + b_0 \,(-\imi\,\omega)\,.
\end{equation}\end{linenomath}
Here $(-\imi \,\omega)^\gamma \coloneqq \abs{\omega}^\gamma \exp(- \imi  \pi \gamma \operatorname{sign}(\omega)/2)$
and $a_0, b_0$ are arbitrary positive constants.
This equation has been considered, for example, in~\cite{szabo1955causal,szabo1994time,waters2000timedomain}.
For tissue, the  exponent $\gamma$ in \eqref{eq:powlaw}
is in the range $(1,2]$. Note that for  positive
$\gamma$ that is not an integer the power law model is weakly causal.
In~\cite{kowar2012photoacoustic} it has been shown that the dissipative waves
modeled by \eqref{eq:powlaw} are strongly causal only if $\gamma\in (0,1)$.
\end{example}

\begin{example}[Model of Kowar,  Scherzer and Bonnefond]
The model proposed by Kowar,  Scherzer and Bonnefond (KSB model)
\cite{kowar2011causality} reads
\begin{linenomath} \begin{equation} \label{eq:kbs}
   \alpha(\omega) = \frac{a_0\,(-\imi\,\omega)}{c_\infty\,
   \sqrt{1 + (-\imi \,\tau_1\,\omega)^{\gamma-1}}}
   + b_0\,(-\imi\,\omega)
   \quad \text{ for } \gamma\in (1,2] \,.
\end{equation}\end{linenomath}
The KBS model is strongly causal. Because strong causality implies weak causality~\cite{kowar2011causality}) the KBS model is also weakly causal.
It satisfies the small frequency approximation
$\Re[ \alpha(\omega)] \asymp a_0
\sin(\frac{\pi}{2}(\gamma-1))/\kl{2\,c_\infty\,\tau_1}
\, \abs{\tau_1\,\omega}^\gamma$ as $\om \to 0$.
Thus \eqref{eq:kbs} behaves  as a power law for small frequencies.
In fact, the KBS has been proposed as a strongly causal alternative to the power
law for the range  $\gamma\in (1,2]$, where the power law fails being
strongly causal.
\end{example}

\begin{example}[Model of Nachman, Smith and Waag]
In the model of Nachman, Smith and Waag (NSW model)
with a single relaxation process, the complex attenuation function
takes the form (see \cite{nachman1990equation})
\begin{linenomath} \begin{equation} \label{eq:nsw}
   \alpha(\omega)
      =  \frac{(-\imi\,\omega)}{c_\infty}\,\left(
      \frac{c_\infty}{c_0}\,\sqrt{  \frac{1 + ({c_0}/{c_\infty})^2\,(-\imi\,\tau_1\,\omega)}{1 + (-\imi\,\tau_1\,\omega) }}   -1 \right) \,.
\end{equation}\end{linenomath}
Equation \eqref{eq:nsw} and its generalization using $N$ relaxation processes
have been derived in~\cite{nachman1990equation} based on sound physical
principles.  The resulting attenuated wave equation is causal and
can even be  reformulated  as differential equation of order $N + 2$.
In~\cite{nachman1990equation,kowar2012photoacoustic} it is shown that the
model~{\eqref{eq:nsw}}  is strongly (and thus weakly) causal provided that
$c_0 < c_\infty$. Then the  wave front speed is bounded from above by $c_\infty$.
We note that the attenuation law (i.e. the real part of $\alpha$ of the NSW
model~\eqref{eq:nsw}), satisfies a power law with  exponent $\gamma = 2$ as $\om \to 0$.
\end{example}

\subsection{The forward operator}

In the sequel, we assume the PA source $\source$ to be supported
in an open set $\Omin$. We assume that measurements are taken on a
piecewise smooth surface $\Gamma \subseteq \partial \Omout$ where
$\Omout \subseteq \R^d$  is an  open set with
$\overline \Omin \subseteq \Omout$.
Further, let $ \tmax \geq \diam\kl{\Omout}/c_0$ denote the
final measurement time and suppose  $\tmin$  is a positive number with
$\tmin c_0 < \dist (\Omin, \Gamma)$.

\begin{definition}[PAT forward operator]
We define the PAT forward operator with weakly causal attenuation law $\alpha$
(see Definition~\ref{def:A}) by
\begin{linenomath} \begin{equation}  \label{eq:Walpha}
\Wo_\al \colon
C_0^\infty(\Omin)  \subseteq
L^2(\Omin) \to  L^2(\Gamma \times (0, \tmax)) \colon
	h \mapsto \rest{p_\al}{\Gamma \times (0, \tmax)},
 \end{equation}\end{linenomath}
where $p_\al$ denotes the causal solution of  \eqref{eq:wavealpha}.
In the case of vanishing attenuation, we   write  $\Wo  \coloneqq \Wo_0$.
\end{definition}

According to  Lemma~\ref{lem:Malpha}, the kernel $m_\al(t,r)$ is smooth for $(t,r)\neq (0,0)$.
Moreover,
\begin{linenomath} \begin{equation} \label{eq:Mal}
    \Mo_\al g(\edot ,t)  \coloneqq \int_0^t  m_\al(t,r) g(\edot ,r) \rmd r
\end{equation}\end{linenomath}
defines a bounded linear operator $\Mo_\al \colon L^2(\Gamma \times (\tmin, \tmax)) \to  L^2(\Gamma \times (0, \tmax))$. The representation of the attenuated pressure in terms of
the un-attenuated pressure given in  Lemma~\ref{lem:Malpha} therefore
shows that $\Wo_\al$ is well defined. Moreover, the following result
shows that it can be extended to a bounded linear operator on $L^2(\Omin)$.

\begin{theorem}[Mapping properties of the PAT forward operator] \label{thm:Walpha}\mbox{}
%The   following assertions hold true:
\begin{enumerate}
%\item\label{thm:Walpha-0} $\Wo_\al$ is linear and  bounded with respect to the $L^2$ norms;
\item\label{thm:Walpha-1} $\Wo_\al$ has a unique bounded extension  $\Wo_\al \colon L^2(\Omin) \to  L^2(\Gamma \times (0, \tmax))$;

\item\label{thm:Walpha-2} $\Wo_\al   = \Mo_\al \circ \Wo$.

% following item not true in general (require $m_\al$
%smooth at origin)
%\item\label{thm:Walpha-3}
%If $\al \neq 0$, then $\Wo_\al  \colon L^2(\Omin) \to L^2(\Gamma \times (0, \tmax))$
%has non-closed range.
\end{enumerate}
\end{theorem}

\begin{proof} \mbox{} \ref{thm:Walpha-1}, \ref{thm:Walpha-2}:
It is known that the operator $\Wo \colon L^2(\Omin) \to  L^2(\Gamma \times (\tmin, \tmax)) \colon h \mapsto \rest{p_0}{\Gamma \times (\tmin, \tmax)}$
is well defined, linear, bounded, injective  and has closed range
(see, for example, \cite{haltmeier2016iterative}).
Now suppose $\al \neq 0$. Because $m_\al$ is smooth on
$\set{(t,r) \neq (0,0)}$, it follows that  $\Mo_\al  \colon L^2(\Gamma \times (\tmin, \tmax)) \to L^2(\Gamma \times (0, \tmax))$ is well defined and bounded. Together with   \eqref{eq:Mpalpha} this gives~\ref{thm:Walpha-2} and implies the  boundedness of $\Wo_\al$. In particular, $\Wo_\al$ has a unique bounded extension  $\Wo_\al \colon L^2(\Omin) \to  L^2(\Gamma \times (0, \tmax))$.
\end{proof}

%\ref{thm:Walpha-3}:  Assume to the contrary that $\Wo_\al$ has closed range.
%From~\ref{thm:Walpha-2} and the boundedness of $\Wo^{-1}$ it follows that
%$\Mo_\al = \Wo_\al  \circ \Wo^{-1} $ has closed range. Since
%$\Mo_\al  \colon L^2(\Gamma \times (\tmin, \tmax)) \to L^2(\Gamma \times (0, \tmax))$ is an integral operator
%with  a smooth kernel this yields a contradiction.

The well known explicit  solution formulas for the standard wave equation give explicit expressions for $\Wo$.  The precise forms of these expression depend on the spatial dimension. For example, in two spatial dimensions
we have
\begin{linenomath} \begin{equation} \label{eq:wave-sol}
\forall  (y,t) \in \Gamma \times (0, \tmax)\colon \quad
\Wo h  \kl{y,t} =
  \frac{1}{2\pi c_0}
  \frac{\partial}{\partial t}
  \int_0^{c_0 t} \int_{\sph^{d-1}}  \frac{r\, h \kl{y +  r \varphi}}{\sqrt{c_0^2 t^2-r^2}}
  \rmd \varphi \rmd r \,.
\end{equation}\end{linenomath}
Together with $\Wo_\al   = \Mo_\al \circ \Wo$
this gives an explicit formula for $\Wo_\al$ that can be
implemented efficiently.

\subsection{The adjoint operator}

Because the forward operator   $\Wo_\al \coloneqq L^2(\Omin) \to
L^2(\Gamma \times (0, \infty))$ is linear and bounded its adjoint
$\Wo_\al^*$ exists and is linear and bounded. In this subsection we give
two expressions for  the adjoint that can be used for the
solution of the inverse problem.

First, we derive an expression  for $\Wo_\al^*$ in the form of an explicit formula.
This representation will be used in our numerical reconstruction algorithm.

\begin{theorem}[Adjoint operator in  \label{thm:ai} integral form]\mbox{}
\begin{enumerate}
\item\label{thm:ai-1} $\Wo_\al^* \colon L^2(\Gamma \times (0, \tmax))  \to
L^2(\Omin)$ is well defined and bounded;
\item\label{thm:ai-2} $\Wo_\al^*    =  \Wo^*  \circ \Mo_\al^*$;
\item\label{thm:ai-3}
$\forall g \in L^2(\Gamma \times (\tmin, \tmax)) \,
\forall t \in (0, \tmax) \colon \Mo_\al^* g (\edot ,r)  = \int_r^\tmax m_\al(t,r) g(\edot ,t) \rmd t$.
\end{enumerate}
\end{theorem}

\begin{proof} \mbox{}
\ref{thm:ai-1}: According to Theorem~\ref{thm:Walpha}~\ref{thm:Walpha-1},
$\Wo_\al \colon
L^2(\Omin) \to  L^2(\Gamma \times (0, \tmax)) $ is linear and bounded. Therefore, its  adjoint
is well defined and bounded, too.

\ref{thm:ai-2}:
Follows from Theorem~\ref{thm:Walpha}~\ref{thm:Walpha-2}.

\ref{thm:ai-3}:
 According to the definition of $\Mo_\al \colon L^2(\Gamma \times (\tmin, \tmax)) \to  L^2(\Gamma \times (0, \tmax))$ we have  $ \Mo_\al  g (y,t)  \coloneqq \int_0^t
  m_\al(t,r) g(y,r) \rmd r$. Therefore
  $\Mo_\al^* \colon L^2(\Gamma \times (0, \tmax)) \to  L^2(\Gamma \times (\tmin, \tmax))$
  is given by $\Mo_\al^* g (y,r)  = \int_r^\tmax m_\al(t,r) g(y,t) \rmd t$.
\end{proof}

Note that the adjoint $\Wo^*$ in the absence of attenuation can be
given by  an explicit expression and  therefore  the attenuated adjoint
 $\Wo_\al^*    =  \Wo^*  \circ  \Mo_\al^*$ is also given by an explicit formula.
The actual expressions for $\Wo^*$  depends on the spatial dimension.
For  example,  in two  spatial dimensions, we have
\begin{linenomath} \begin{equation}  \label{eq:ad-sol}
\forall x \in \Omin \colon \quad
\kl{\Wo^* g}\kl{x}
    =
    -\frac{1}{2 \pi}
    \int_{\Gamma}
    \int_{\sabs{x-y}}^\tmax
    \frac{\partial_t g \kl{y,t}}
    {\sqrt{c_0^2 t^2 - \sabs{x-y}^2}}   \, \rmd t  \, \ds(y) \,.
\end{equation}\end{linenomath}

Our next results show that the adjoint operator can additionally
be described by an attenuated wave equation.
In  absence of attenuation, similar formulations for the adjoint have been derived in \cite{arridge2016adjoint,belhachmi2016direct,haltmeier2016iterative}.
For that purpose, we denote by   $\delta_{\Gamma} $ the tempered
distribution on $\R^d \times \R$ defined by
$\inner{\delta_{\Gamma}}{\phi} = \int_{\R} \int_{\Gamma} \phi (x,t) \, \ds(x) \, dt
$ for $\phi \in C_0^\infty(\R^d \times \R)$. Furthermore, we denote by  $\Do_\al^\ast$
the formal $L^2$-adjoint of $\Do_\al$ given by the time convolution
with  the time reversed kernel $\ifourier_t(\al)(-t)$.

\begin{theorem}[Adjoint operator in wave equation form]
Let \label{thm:adjointw1} $\al$ be any weakly causal attenuation
function. For $g \in C_0^\infty(\Gamma \times (0,\tmax)) $ let  $q_\al$ be the solution of the adjoint attenuated wave equation
\begin{linenomath} \begin{equation}   \label{eq:attenuated1}
\kl{ \Do^*_\al  -
\frac{1}{c_0}\frac{\partial}{\partial t }  }^2  q_\al(x,t)
-
\Delta q_\al(x,t)
=  - \delta_{\Gamma}(x) \, g(x,t)
\quad
\text{ on } \R^d \times \R \,,
 \end{equation}\end{linenomath}
with $q_\al (\edot ,t) = 0$ for $t > \tmax$. Then,
\begin{linenomath} \begin{equation}  \label{eq:adjointw1}
\Wo^*_\alpha (g)  =  \frac{\partial  q_\al}{\partial t } (\edot,0)
 \,.
\end{equation}\end{linenomath}
 \end{theorem}

\begin{proof}
See Appendix~\ref{ap:proof}.
\end{proof}

From Theorem~\ref{thm:adjointw1}, we immediately obtain the following
alternative form.

\begin{corollary}[Adjoint operator in time reversed wave equation form]
Suppose the assumptions of Theorem~\ref{thm:adjointw1} \label{thm:adjointw2}  are satisfied and let $q_\al^*$
be the causal solution of
\begin{linenomath} \begin{equation}  \label{eq:attenuated2}
\kl{ \Do_\al  +
\frac{1}{c_0}\frac{\partial}{\partial t }  }^2  q_\al^*(x,t)  - \Delta q_\al^*(x,t)  =  - \delta_{\Gamma}(x) \, g(x,\tmax-t)
\quad \text{ on } \R^d \times \R \,.
\end{equation}\end{linenomath}
Then we have
\begin{linenomath} \begin{equation}  \label{eq:adjointw2}
	\Wo^*_\alpha (g)  =  - \frac{\partial  q_\al^*}{\partial t } (\edot,\tmax) \,.
\end{equation}\end{linenomath}
\end{corollary}

\begin{proof}
Follows from Theorem~\ref{thm:adjointw1}  with
$q_\al^*(x,t) = q_\al(x,\tmax-t)$.
\end{proof}

Note that~\eqref{eq:attenuated1}, \eqref{eq:adjointw1} and \eqref{eq:attenuated2}, \eqref{eq:adjointw2} have
a similar form  to the time reversal used in \cite{kalimeris2013photoacoustic,ammari2011time,kowar2014time,kowar2014time2,burgholzer2007compensation,treeby2010photoacoustic}. However, unlike  the time-reversed wave equation, where the corresponding wave blows up, the adjoint formulation has the same stability properties as the forward equation. Therefore, in contrast to  the time reversal procedure, there is no need to include regularization to implement~\eqref{eq:attenuated1}  or \eqref{eq:attenuated2}.
Accurate numerical solution of dissipative wave equations for realistic parameters is challenging and numerically quite
expensive.  Let us consider this issue for the wave equation of Nachmann, Smith and Waag with only one relaxation process. The relaxation time $\tau_1$ for fluids is about $\unit[100]{ns} $ and the discretization step size  should be at least close the relaxation time. This results in a  time discretization much finer
than usually employed for the simulation of un-attenuated waves, and therefore  an increased numerical cost.

\section{Solution of the inverse problem}
\label{sec:inverse}

In this section we solve the inverse problems of PAT with attenuation
using iterative regularization methods. Our method comes with a
clear convergence analysis. We further present details on its actual
implementation. Throughout the rest of this paper, we write
$\enorm{}$ for the regular $L^2$-norms on $L^2(\Omin)$ and
$L^2(\Gamma \times (0, \infty))$ as well as for the
operator norm between these two spaces.

\subsection{The Landweber method}
\label{sec:landweber}

The Landweber method for the solution of  $\Wo_\al \source \simeq   g^\delta$ is defined by
\begin{linenomath} \begin{equation}  \label{eq:landweber}
    \forall n\in \N \colon \quad
    h_{n+1}^\delta  =
    h_n^\delta -  \la    \Wo_\al^* \kl{ \Wo_\alpha h_n^\delta  -  g^\delta } \,.
\end{equation}\end{linenomath}
Here $g^\delta$ are the noisy data,
$0 < \la \leq \snorm{\Wo_\al}^{-2} $ is the step size,
and $h^\delta_0  \coloneqq h^\delta$  some  initial guess.
The superscript  $\delta$ indicates the noise level, which means
that an  estimate $\norm{\Wo_\al \source^\star  - g^\delta} \leq \delta$
is available, where  $\source^\star$ is the unknown true solution.

The Landweber iteration will be combined with Morozov's discrepancy
principle.  According to the discrepancy principle,  the iteration is terminated at the index
 $n = n(\delta, y^\delta)$, when for the first time
\begin{linenomath} \begin{equation}  \label{eq:dep}
	\norm{h_{n+1}^\delta  -  g^\delta }  \leq \tau \delta
\end{equation}\end{linenomath}
with some fixed $\tau >1$.
From Theorem \ref{thm:Walpha}  and the general theory of
iterative regularization methods, we obtain  the following result.

\begin{theorem}[Convergence of the  Landweber iteration]
Suppose $h \in L^2(\Omin)$, $\eps >0$ and let $g^\delta \in L^2((0, \tmax) \times \Gamma)$ satisfy $\snorm{g^\delta - \Wo_\al h}  \leq \delta$.

\begin{enumerate}
\item \emph{Exact data:} If $\delta=0$, then  $(h_n)_{n \in \N}$ strongly converges to the
$h$.

\item \emph{Noisy data:}
Let $(\delta_m)_{m\in \N} \in (0, \infty)^\N$ converge  to zero and let $(h_m)_{m\in \N} \in L^2((0,\tmax)\times \Gamma)^\N$
satisfy
$\snorm{h_m- \Wo_\al h} \leq \delta_m$.
Then the following hold:
\begin{itemize}
\item The stopping indices
 $n_m := n_*(\delta_m, h_m) $  are  well defined by  \eqref{eq:dep};
 \item
 We have  $\snorm{h_{n_m}^{\delta_m} - h} \to 0 $ as $m\to \infty$.
\end{itemize}
\end{enumerate}
\end{theorem}

\begin{proof}
According to the Theorem~\ref{thm:Walpha}, the operator $\Wo_\al$  is bounded.
The claims therefore follow  from standard  results for iterative regularization methods
(see, for example, \cite{engl1996regularization,kaltenbacher2008iterative}).
\end{proof}

The Landweber method is the most basic iterative regularization method
for the solution of inverse problems and behaves  very stable due to the smoothing effect
of the adjoint. On the other it is
quite slow in some applications.  Accelerated version include $\nu$-methods \cite{brakhage1987ill,engl1996regularization},
the CG Algorithm \cite{kammerer1972convergence,hanke1995conjugate},
preconditioned Landweber iterations \cite{egger2005preconditioning}
or Kaczmarz-type iterations \cite{haltmeier2007kaczmarz,haltmeier2009convergence}.
Here we chose  the Landweber iteration  because the  main  aim of  the present paper is  demonstrating  the effectiveness of iterative  methods  for PAT with acoustic attenuation.
Furthermore, for  the considered  application already 10 iterative steps provide very accurate results. Nevertheless, note that generalization to other iterative regularization methods
such as the  steepest descent  or the conjugate gradient  method is straight forward.

Another advantage of the Landweber method is that it can easily be combined with
 a projection step to improve performance. The resulting  projected Landweber method reads
\begin{linenomath} \begin{equation}  \label{eq:landweberP}
   	h_{n+1}^\delta  =
	\Po_C
	\kl{ h_n^\delta -  \la    \Wo_\al^* \kl{ \Wo_\alpha h_n^\delta  -  g^\delta } }\,,
\end{equation}\end{linenomath}
where  $\Po_C$ denotes the  projection on a closed convex set
$C \subseteq L^2(\Omout)$,
Actually, in our numerical implementation, we use the projected
Landweber method  with $C$ being  the cone of non-negative functions,
which turned out  to produce slightly better  results than the pure
Landweber method with a comparable numerical complexity.

\subsection{Implementation of the (projected)  Landweber iteration}

We outline the implementation for the case that
$\Gamma \coloneqq \partial B_\RR(0)$ is a circle of radius $\RR$ in two spatial dimensions. Extension for more general geometries and higher dimensions are straight forward.
Our approach uses the relations   $\Wo_\al      =  \Wo \circ \Mo_\al $ and $\Wo_\al^\ast  = \Mo_\al^* \circ \Wo^*$ (see Theorems~\ref{thm:Walpha} and Theorems~\ref{thm:ai})
 that relate the attenuated pressure to the un-attenuated pressure in the direct and adjoint problems, respectively.

For that purpose the PA source $h \colon \R^2 \to \R$ is represented by a discrete vector $\hnum \in \R^{(N_x+1) \times (N_x+1)} $ obtained by uniform sampling
\begin{linenomath}\begin{align}\label{eq:hnum1}
 h[i] &\simeq h(x_{i})
&& \text{  for $i  = (i_1, i_2)\in \{0, \dots, N_x\}^2$ }
 \\ \label{eq:hnum2}
x_i &= (-R,-R) + i \; \frac{ 2R}{N}
&& \text{  for $i  = (i_1, i_2)\in \{0, \dots, N_x\}^2$ }
 \,.
\end{align}\end{linenomath} Here  $(N_x+1)^2$ is the number of spatial discretization points on an equidistant
Cartesian grid. Further,  any function $g \colon \partial \Omout \times [0, \tmax] \to \R$
is discretely represented by a vector
$\gnum \in \R^{N_\ph \times (N_t+1)}$, with
\begin{linenomath}\begin{align}\label{eq:gnum1}
\gnum[k,\ell]  &\simeq   g( \RR (\cos \varphi_k,\sin \varphi_k), t_\ell  )
&& \text{for }  (k,\ell)  \in \set{0, \dots, N_\varphi -1} \times \set{0, \dots, N_t} \,,
\\ \label{eq:gnum2}
	\varphi_k
	& \coloneqq  k \, \frac{2\pi}{N_\varphi}
	&& \text{for }  k \in\set{ 0, \dots, N_\varphi -1 }\,,
\\ \label{eq:gnum3}
	t_\ell
	&\coloneqq  \ell \, \frac{2 R}{N_t}\,
	&& \text{for } \ell  \in\set{  0, \dots, N_t }\,.
\end{align}\end{linenomath}
 Here $N_\ph$  is the number  of angular samples (detector locations) and
$N_t + 1$ the number of temporal samples. The sampling conditions
obtained  in \cite{haltmeier2016sampling} imply that   $\Delta x  \simeq c_0 \, \Delta t \simeq \RR \,\Delta \ph$, where $\Delta x \coloneqq \sfrac{2\RR}{N_x}$
$\Delta t \coloneqq \sfrac{\tmax}{N_t}$ and
$\Delta \ph \coloneqq \sfrac{2\pi}{N_\ph}$ yield aliasing free sampling.

The Landweber iteration \eqref{eq:landweber} and its projected version
\eqref{eq:landweberP} are implemented by
replacing $\Wo$, $\Wo^*$, $\Mo_\al$ and $\Po_C$ with discrete counterparts
 \begin{linenomath}\begin{align}\label{eq:Wn}
	 &\Wnum
	 \colon \R^{(N_x+1) \times (N_x+1) }  \to \R^{N_\ph \times (N_t+1)} \,,
 \\ \label{eq:Mn}
 	 &\Mnum_\al
	 \colon  \R^{N_\ph \times (N_t+1)}  \to \R^{N_\ph \times (N_t+1)} \,,
\\ \label{eq:Bn}
 	 &\Bnum
	 \colon  \R^{N_\ph \times (N_t+1)}  \to \R^{(N_x+1) \times (N_x+1) } \,,
\\ \label{eq:Pn}
 	 &\Pnum_{\Cnum}
	 \colon  \R^{(N_x+1) \times (N_x+1) }   \to \R^{(N_x+1) \times (N_x+1) } \,.
 	 \end{align}\end{linenomath} The resulting discrete (projected) Landweber method is then defined by
\begin{linenomath} \begin{equation}  \label{eq:landweber-d}
	\forall n \in \N \colon \quad
	\hnum_{n+1}^\delta  =
	\Pnum_{\Cnum} \kl{\hnum_n^\delta -
	\la \;  \Bnum  \Mnum_\al^*
	\kl{ \Mnum_\alpha    \Wnum    \hnum_n^\delta  -  \gnum^\delta }
	} \,.
\end{equation}\end{linenomath}
Note that for the sake of computational efficiency the operator $\Bnum$ will be implemented by a filtered backprojection procedure that which is the exact discrete adjoint of  $\Wnum$.
On the other hand, as numerical approximation of $ \Mo_\al^*$ we take the exact
adjoint of the discretization of $\Mnum_\al$.  Finally, the discrete projection
will simply  be taken as $\Pnum_{\Cnum} (\hnum) \coloneqq \operatorname{max} \set{0, \hnum}$, which is the projection on convex cone $\Cnum \coloneqq  [0, \infty)^{(N_x+1) \times (N_x+1) }  \subseteq  \R^{(N_x+1) \times (N_x+1) }$.  How to implement the other operators will be
described in the following subsections.

\begin{remark}[Numerical complexity]
Under the reasonable assumption $N_x\sim N_\ph \sim N_t$, one
iterative step \eqref{eq:landweber-d} requires $\mathcal{O} (N_x^3)$
floating point operations (FLOPS) with small leading constants and comparable
to the  effort of a standard FBP reconstruction algorithm.
Since a small number of around 10 turned out to be sufficient, our algorithm
is numerically  quite efficient.
\end{remark}

\subsection{Implementation of $\Wo$ and its adjoint}

For numerically approximating the un-attenuated wave operator $\Wo$, we
discretize the explicit formula~\eqref{eq:wave-sol}.
For that purpose, we write
\eqref{eq:wave-sol} in the form $\Wo h = c_0^{-1} \partial_t  \Ao  \Mo h$,
where
\begin{linenomath}\begin{align}\label{eq:means}
\Mo h  \kl{y,r}
&\coloneqq
  \frac{1}{2\pi}
  \int_0^{2\pi}  h \kl{y + r (\cos \beta, \sin \beta)}
  \rmd \beta \,,  \\
  \Ao g  \kl{y,t}
  &\coloneqq
  \int_0^{c_0t} \frac{r \, g (y,r)}{\sqrt{c_0^2t^2-r^2}}
  \rmd r
\end{align}\end{linenomath} for $(y,t) \in \Gamma \times (0, \tmax)$. The operator $\Mo$ is the spherical mean Radon transform and $\Ao$ the Abel transform (in the second component).
 We compute discrete  spherical means  by approximately evaluating   \eqref{eq:means} at the
all discretization points $(\RR \ph_k, c_0 t_j)$ using the trapezoidal
rule for discretizing the integral over $\beta$. The
values  of $h$ for applying the trapezoidal  rule  are   obtained  by
using bilinear interpolation of  $\hnum$.   Next for any $k$, the Abel transform is approximately computed by replacing  $g(y_k, \edot)$  with the linear interplant  through the data pairs $(c_0 t_\ell, g(y_k, c_0 t_\ell))$. Finally,
we approximate the time derivative $\partial_t$ by finite differences.
Inserting these approximations  to $\Wo  = c_0^{-1} \partial_t  \Ao  \Mo $
yields the discretization $\Wnum$.

The adjoint wave operator $\Wo^*$ is implemented in a similar manner using   \eqref{eq:ad-sol}  which can be written in the form $\Wo  = - c_0^{-1}
 \Mo^* \Ao^*  \partial_t$.  The operators  $\Ao^*$ and $ \partial_t$ are discretized as above.
 The adjoint $\Mo^*$ of the spherical mean operator is implemented using a backprojection procedure described in detail
 in~\cite{burgholzer2007temporal,finch2007inversion}.

\subsection{Implementation  of $\Mo_\al$}

Recall that for any $(x,t) \in \Gamma \times [0, \tmax]$, we have
\begin{linenomath} \begin{align}\label{eq:M1}
     \Mo_\al g(x,t)
     &= \int_0^{\tmax} m_\al(t,r)\,g(x,r)\, \rmd r \,,
     \\ \label{eq:M2}
     \fourier_t (m_\al  (\edot,r)  )(\om)
     &= \frac{\om}{ \om/c_0 + \imi \al(\om)}
\, e^{\imi (\om / c_0  + \imi \al(\om)) \abs{r}} \,.
     \end{align}\end{linenomath}
     The operator $\Mo_\al$ is discretized based on these relations by approximately
computing $m_\al(t_{\ell}, t_{\ell'})$ using~\eqref{eq:M2} and then discretizing~\eqref{eq:M1}.
This yields the discrete approximation
\begin{linenomath}\begin{align}\label{eq:Mn2}
     (\Mnum_\al \gnum)[k,\ell]
     & \coloneqq
     \Delta t \;
     \sum_{\ell'=0}^{N_t} \mnum_\al[\ell,\ell'] \, \gnum [k,\ell']   \,,
     \\ \label{eq:Mn3}
     \operatorname{FFT}(\mnum_\al)[\ell,\ell']
     & \coloneqq
      \frac{\om[\ell]}{ \om[i]/c_0 + \imi \al(\om[\ell])}
\, e^{\imi (\om[\ell] / c_0  + \imi \al(\om[\ell])) \abs{r[\ell']}} \,.
\end{align}\end{linenomath} Here $\operatorname{FFT}$ denotes the discrete Fourier transform in the
first component and the discrete kernel $\mnum_\al[\ell,\ell']$ is obtained
by   applying the  inverse fast Fourier  transform in the first component.
Moreover, $\om[\ell] = \ell  \Delta \om + N_t \pi / \tmax$ with $\Delta \om = 2\pi / \tmax$.
Finally, the adjoint $\Mo_\al^\ast$ is implemented by the adjoint
$(\Mnum_\al^\ast \gnum)[k,\ell]
     \coloneqq
     \Delta t \;
     \sum_{\ell'=0}^{N_t} \mnum_\al[\ell',\ell] \, \gnum [k,\ell']$
of the discrete operator~\eqref{eq:Mn2},~\eqref{eq:Mn3}.

\begin{figure}[htb!]
\includegraphics[width=0.495\textwidth]{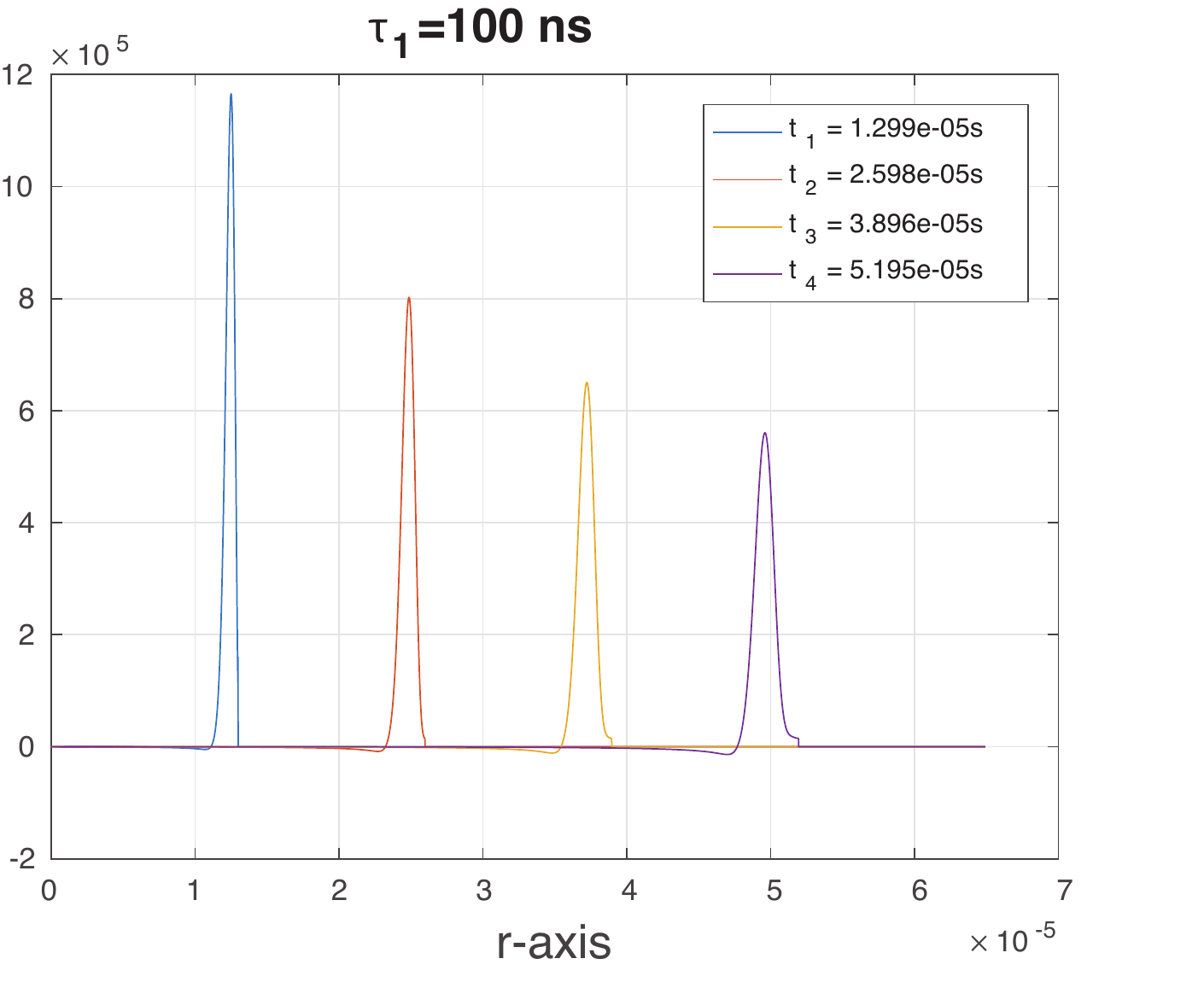}
\includegraphics[width=0.495\textwidth]{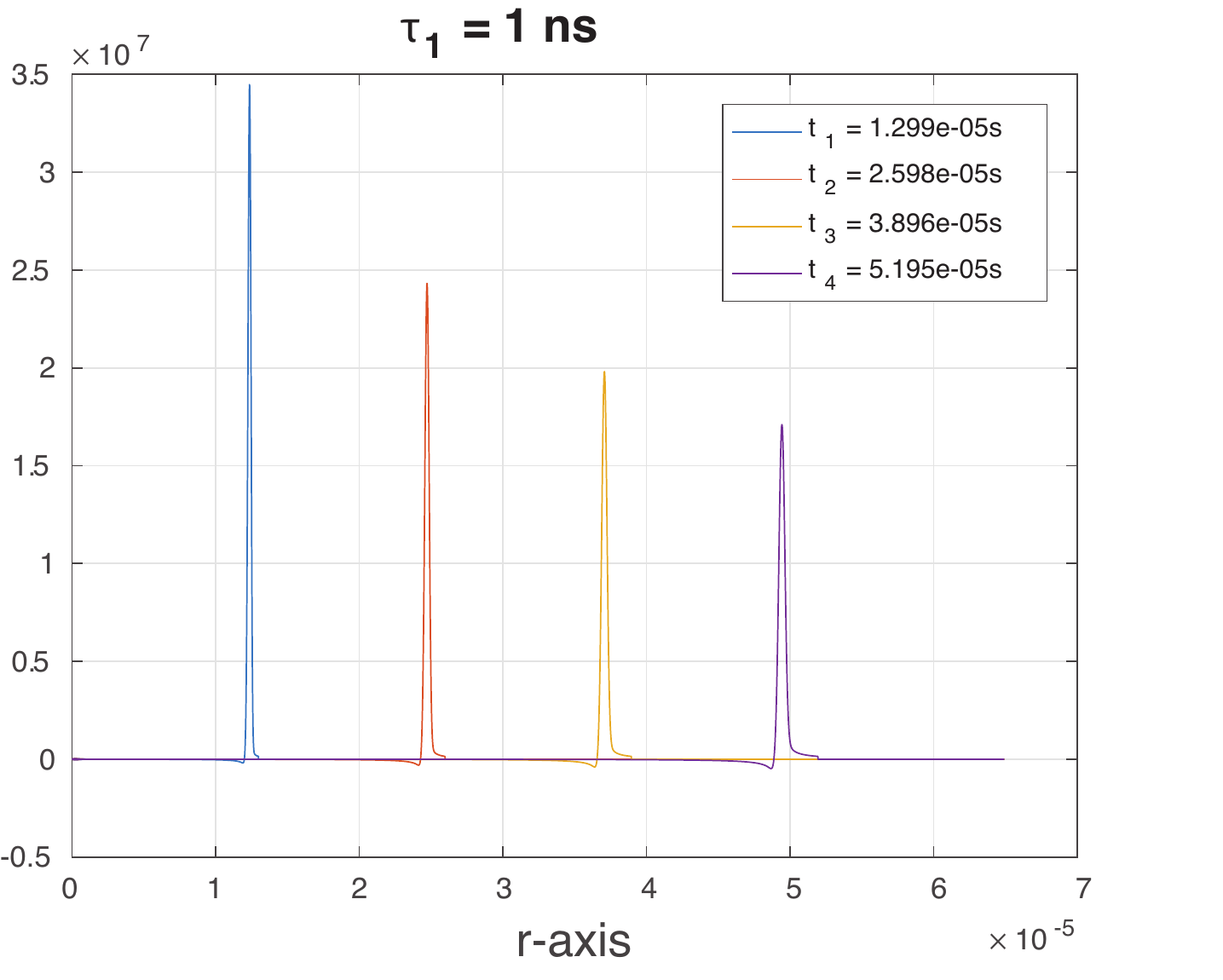}
\caption{\textsc{Visualization of the kernel $m_\al(t, \edot )$.} Left: Kernel using relaxation time
$\tau_1 = \unit[100]{ns}$ (strong attenuation). Right:
Kernel using  relaxation time  $\tau_1 = \unit[1]{ns}$
(weak attenuation).}
\label{fig:kernel}
\end{figure}

\section{Numerical results}
\label{sec:num}

In this section, we present numerical simulations for
PAT with and without  attenuation. For all numerical results
presented below, the region $\Omout$ is a disc of radius $\RR$.
The final measurement time is taken as $\tmax = 2 \RR / c_0$,  where $c_0 = \unit[1540]{m/s}$ is taken as the
 sound speed in water. For all reconstruction results we take
 $ N_x  =  N_t =  N_\ph$ in the reconstruction.
 In order to avoid inverse crime, the data have been computed using
 a finer temporal discretization.

\subsection{Pressure simulation}

For the reconstruction results using attenuation data,  we will employ the NSW model. It has quadratic frequency dependence for small frequencies.
This describes attenuation  of water that has an exponent close to $2$ for small frequencies~\cite{kowar2012photoacoustic,nachman1990equation}.
We use $c_\infty = \unit[1623]{m/s}$. For simplicity, we restrict ourselves to a single  relaxation process  in the NSW model. For the relaxation time $\tau_1$, we consider  two
cases, for which we also consider different radii of the measurement circle:
\begin{itemize}
\item Case 1:  $\RR = \unit[5]{cm}$ and $\tau_1 = \unit[100]{ns}$;
\item Case 2:  $\RR = \unit[5]{mm}$ and $\tau_1 = \unit[1]{ns}$.
\end{itemize}
Figure~\ref{fig:kernel} shows the corresponding kernel function
$m_\al(t, \edot )$ for the above different relaxation times. We see that the
support of $m_\al(t, \edot )$ increases with the relaxation time indicating
that the  larger relaxation time corresponds to stronger attenuation.

\begin{figure}[htb!]
\includegraphics[width=0.495\textwidth]{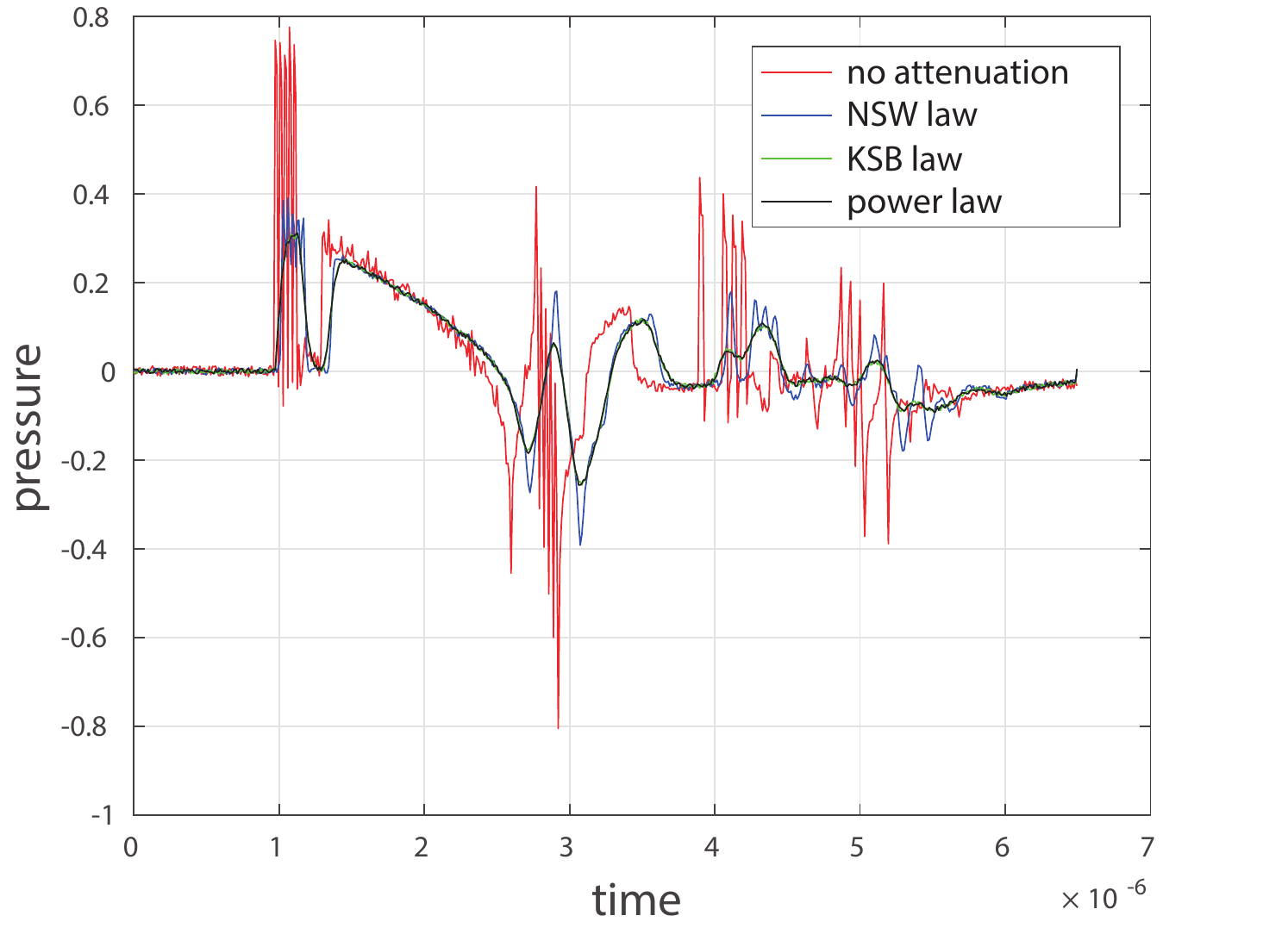}
\includegraphics[width=0.495\textwidth]{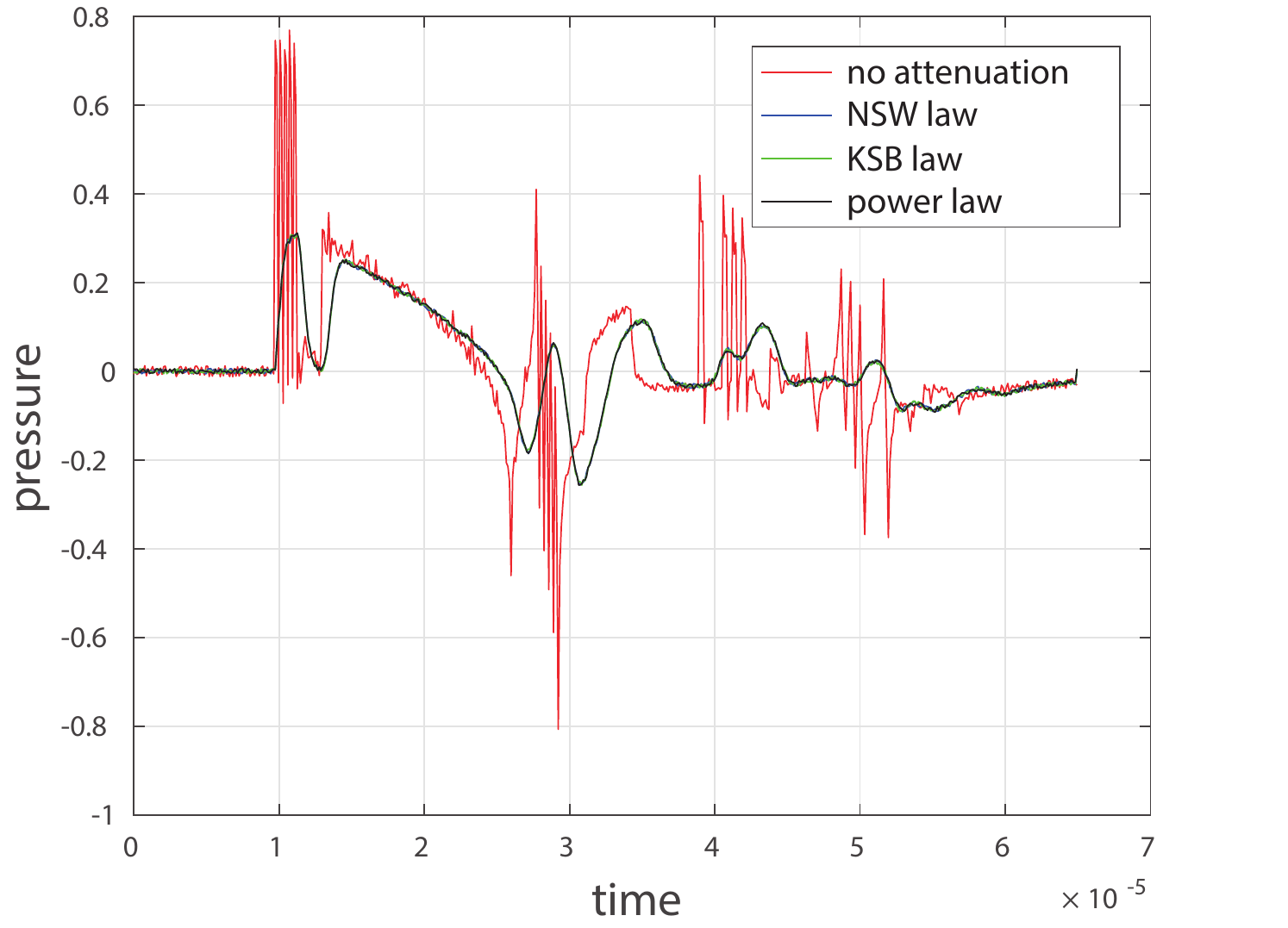}
\caption{\textsc{Simulated un-attenuated
\label{fig:data} and attenuated pressure data.}
Left: Weak attenuation case  $\tau_1 = \unit[100]{ns}$.
Right:  Strong attenuation case $\tau_1 = \unit[1]{ns}$.}
\end{figure}

In Figure~\ref{fig:data}, we compare the (noisy) un-attenuated pressure
data measured $p_0(\x, \edot )$ at location  $\x= (\RR,0)$ with  (noisy) attenuated pressure data $p_\al(\x, \edot )$ according to the NSW model  for the phantom shown in Figure~\ref{fig:strong}. We also  compare it to the pressure data obtained with the KSB model and the power law with exponent $\gamma=2$. The parameter settings of the KSB and the power law models have been chosen such that the real and imaginary parts of the complex attenuation laws are as close as possible to the one of the NSW law for small frequencies.
For the strong attenuation case (Figure~\ref{fig:data}, right), we see that all attenuated pressure data are very similar. Indeed, if we simulate noisy data via the NSW model and then estimate the initial data via the power law, the KSB law or the NSW law, then the results would  hardly be distinguishable. However, the left picture in Figure~\ref{fig:data} shows that this is not true in the small attenuation case where the different attenuation laws yield quite different attenuated pressure signals.
Note that for the power law we actually implemented a causal form, where we
have truncated $m_\alpha(t,r ) $ for $r>t$ after evaluating~\eqref{eq:M2}.

\subsection{Reconstruction results for strong attenuation}

The numerical simulations that we present in this subsection correspond to strong attenuation with a relaxation time $\tau_1 =  \unit[100]{ns}$. The radius of the region of interest is taken as $\RR = \unit[5]{cm}$
and we take $N_x = N_t = N_\ph=600$.

\begin{figure}[htb!] \centering
\includegraphics[width=0.495\textwidth]{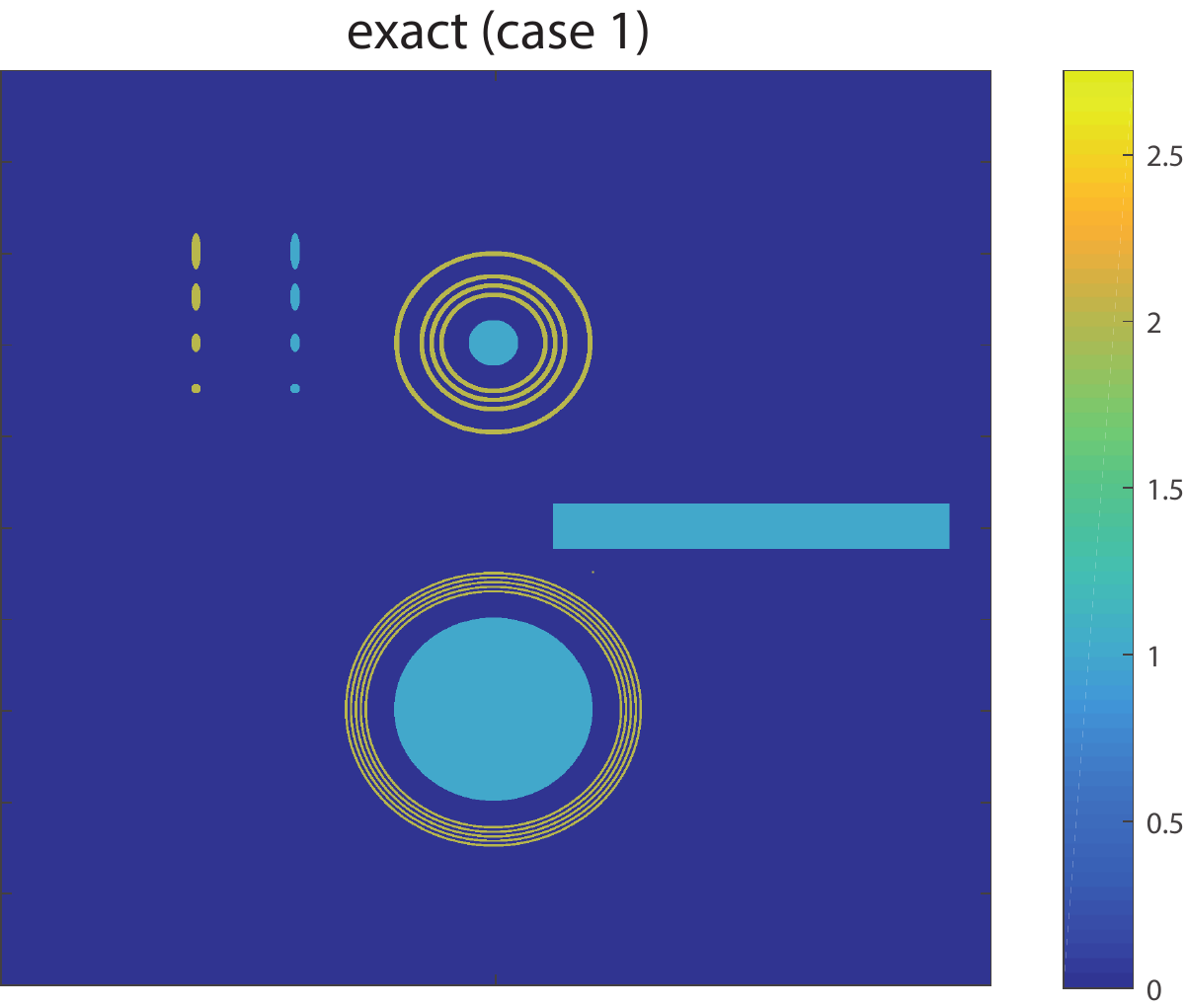}
\includegraphics[width=0.495\textwidth]{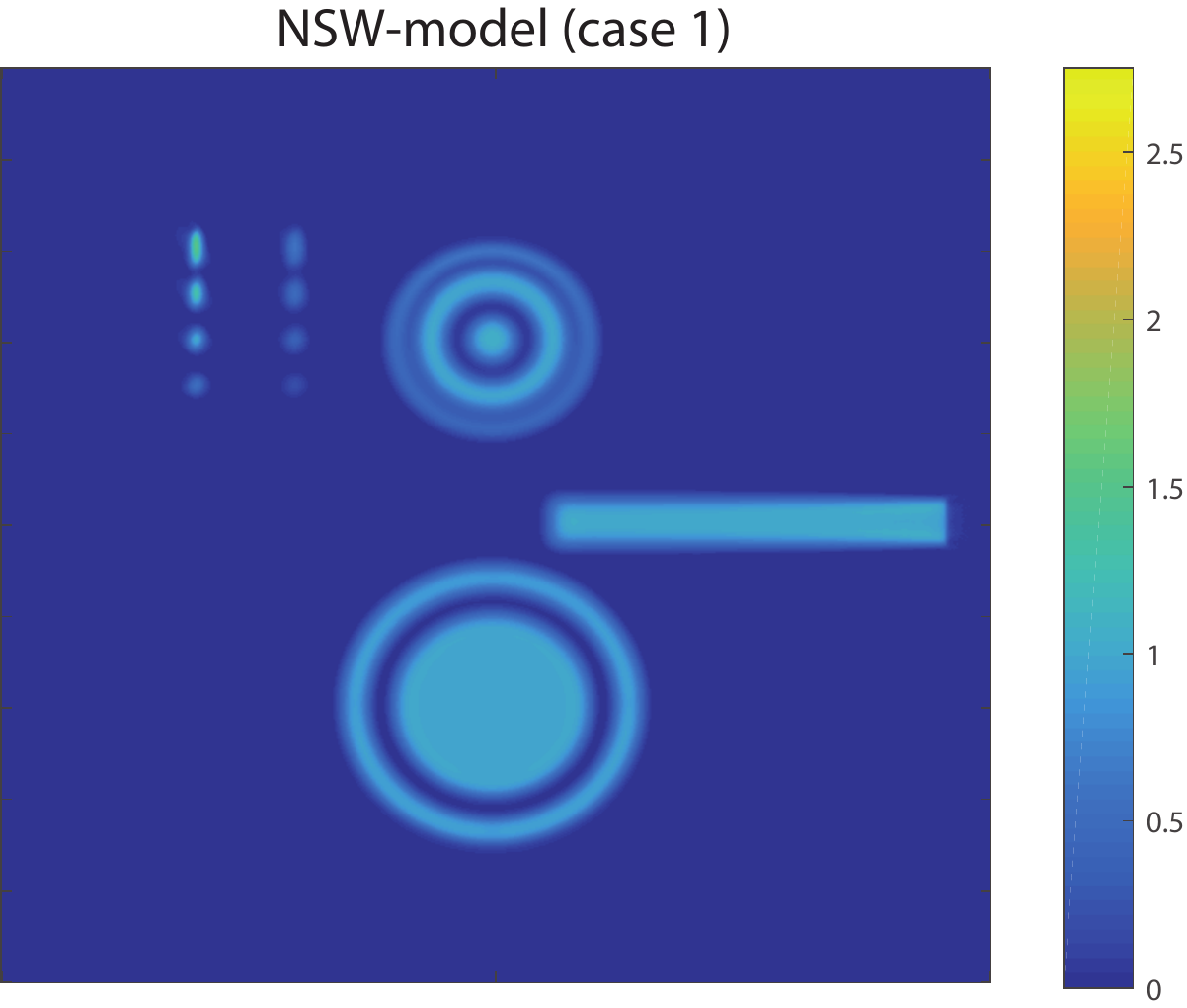}\\
\includegraphics[width=0.495\textwidth]{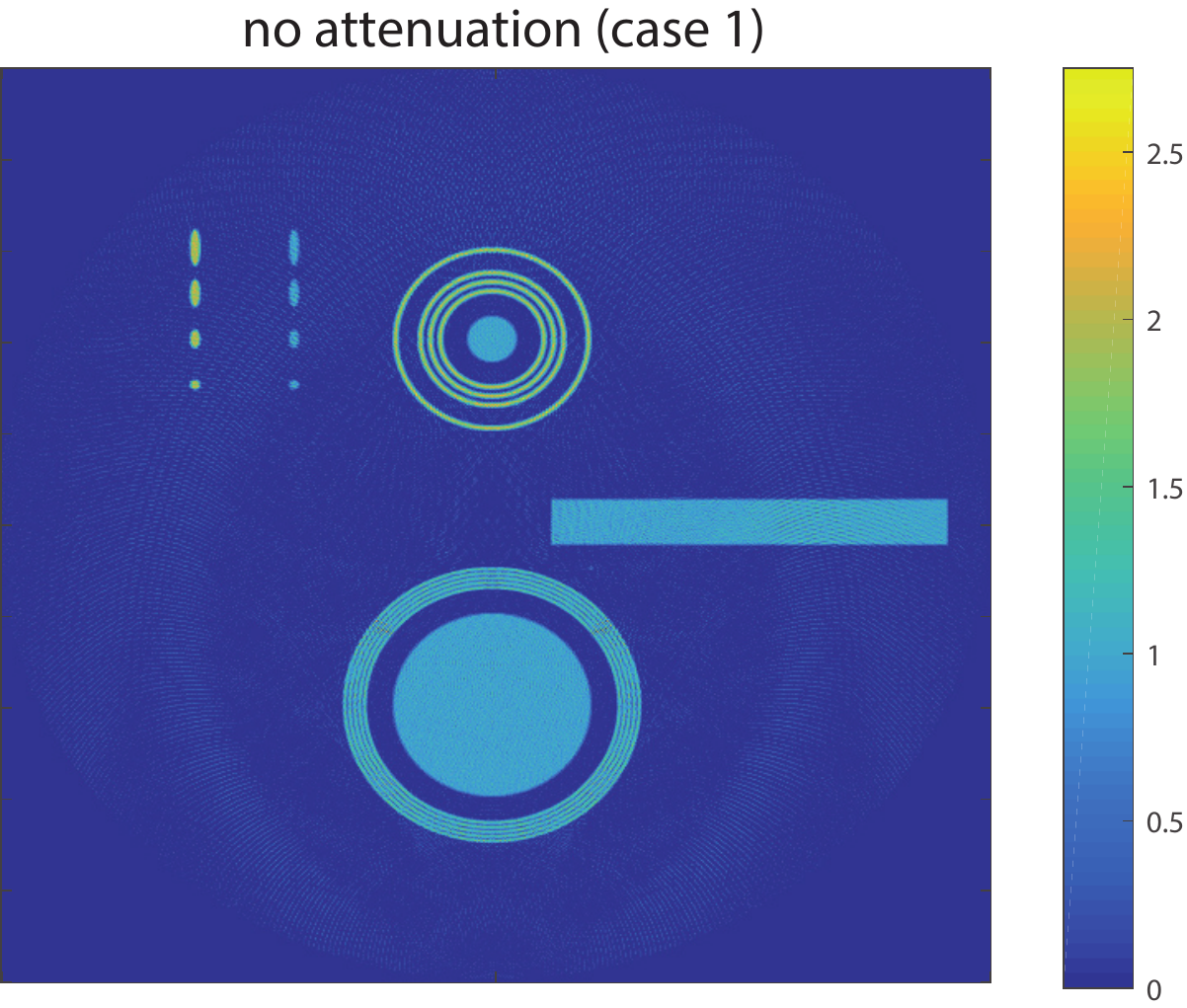}
\includegraphics[width=0.495\textwidth]{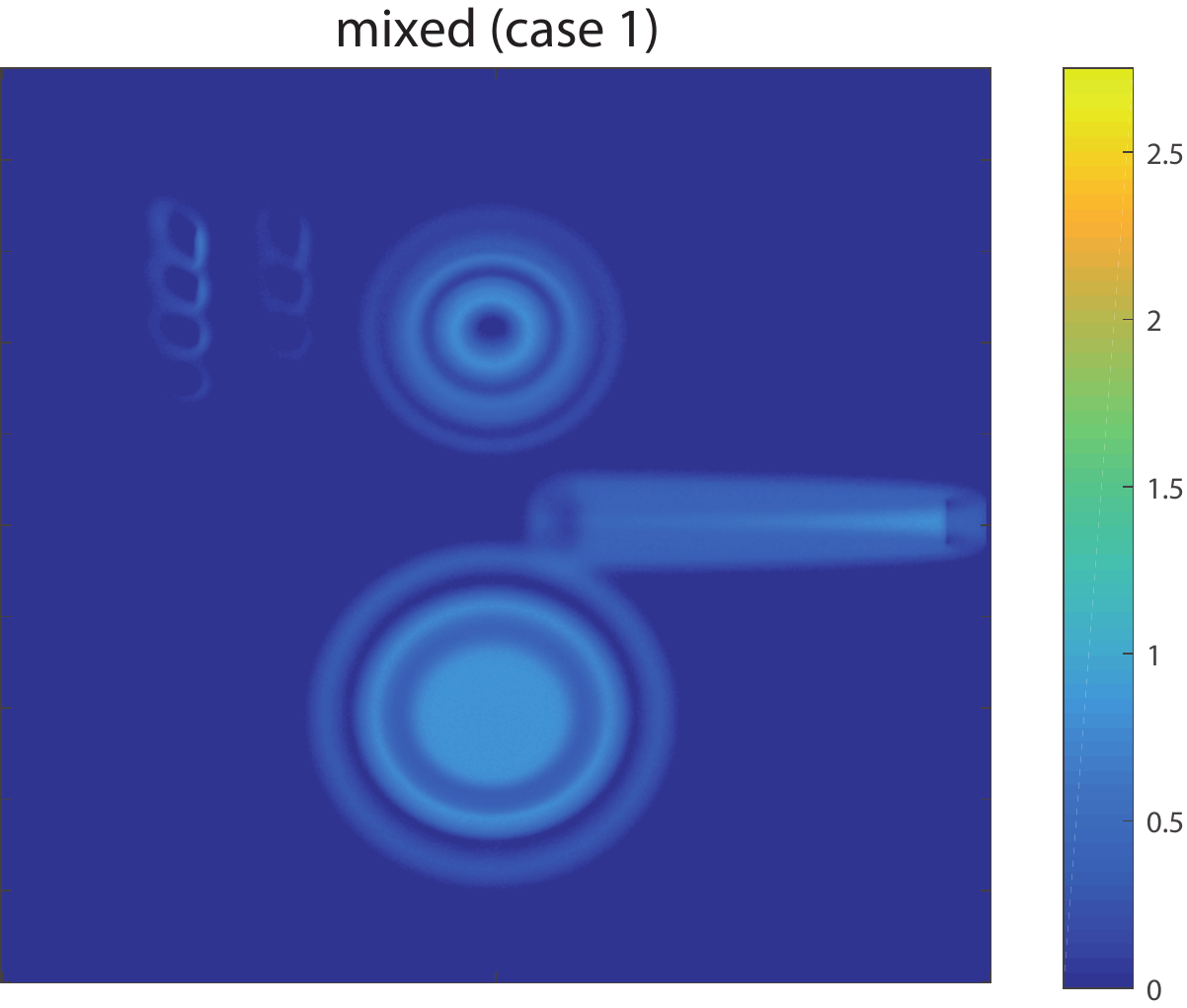}
\caption{\textsc{Reconstructions in the strong attenuation case  $\tau_1 = \unit[100]{ns}$.}
Top left: Exact PA source. Top right:  Reconstruction based on the
NSW model. Bottom left: Reconstruction in the absence of attenuation. Bottom right: Reconstruction from attenuated data  but neglecting
attenuation in the reconstruction.}
\label{fig:strong}
\end{figure}

The numerical phantom  and the numerical results using the projected Landweber iteration with and without attenuation are presented in
Figure~\ref{fig:strong}. We see that the reconstructions using the NSW model (top right)
yields a smoother results  than in the absence of attenuation (bottom left).
In the case with attenuation the thin concentric annuli cannot be resolved, they appear as single thick blurred annulus. Moreover, small or thin structures are blurred and provide less contrast  in the case of attenuation. We also applied the  projected Landweber iteration using the un-attenuated  wave equation to the attenuated data. The reconstruction shown in the bottom right image in Figure~\ref{fig:strong} indicates that thin and long structures are strongly blurred and displaced. Actually, details with small diameter cannot be estimated reliably which
clearly indicates that attenuation has to be taken into account.
This also reflects  that attenuated data are not only smoothed but also displaced. The artifacts in the mixed reconstruction might be reduced by shifting the pressure data  appropriately. Indeed, heuristic rules  performing such a shift are often applied in practice.  However, as the shift depends on the location  and the frequency content of the object  applying a reasonable shift
seems a non-trivial issue that is not required at all in our approach.

\begin{figure}[htb!]
\includegraphics[width=0.5\textwidth]{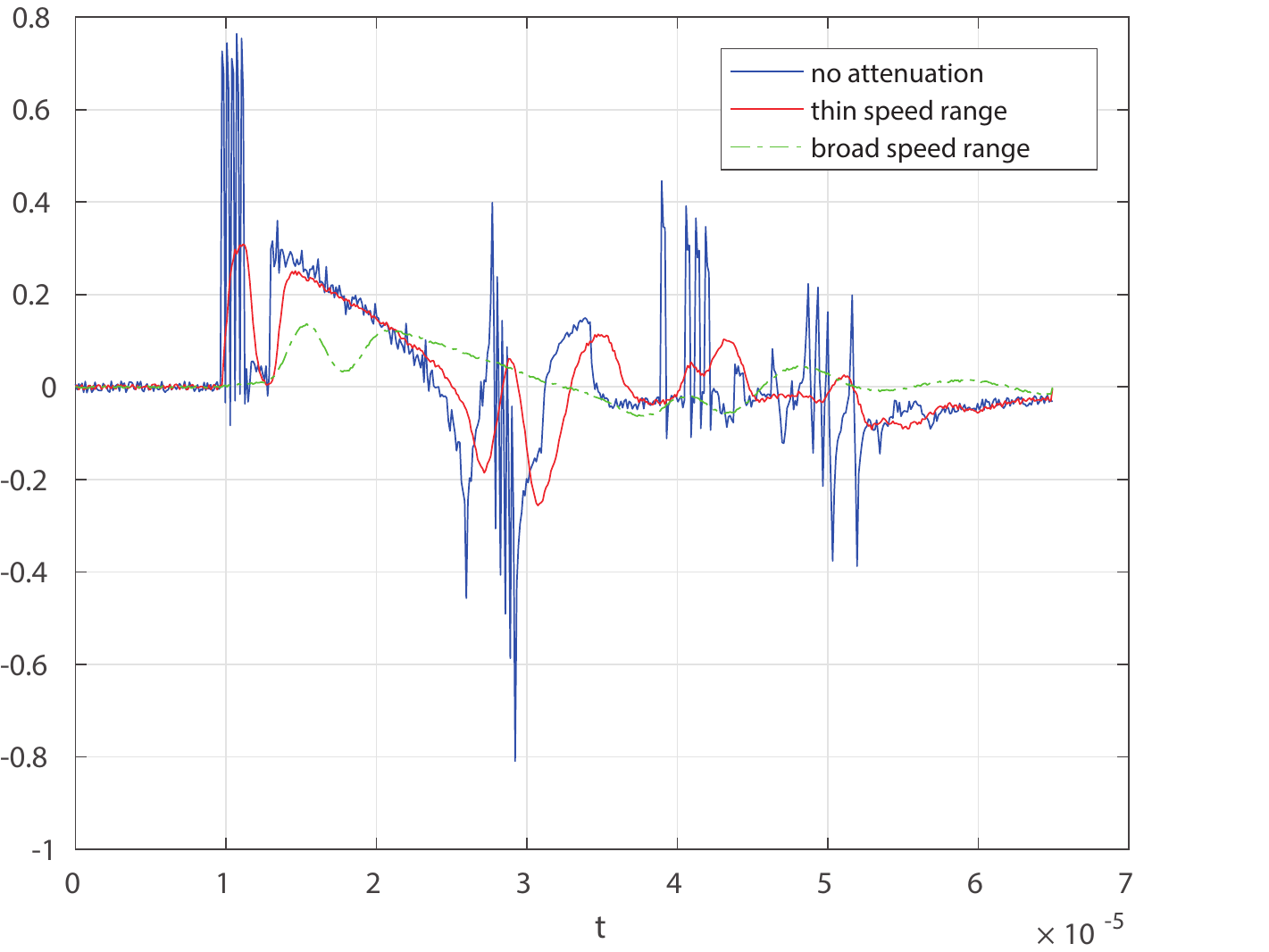}
\includegraphics[width=0.49\textwidth]{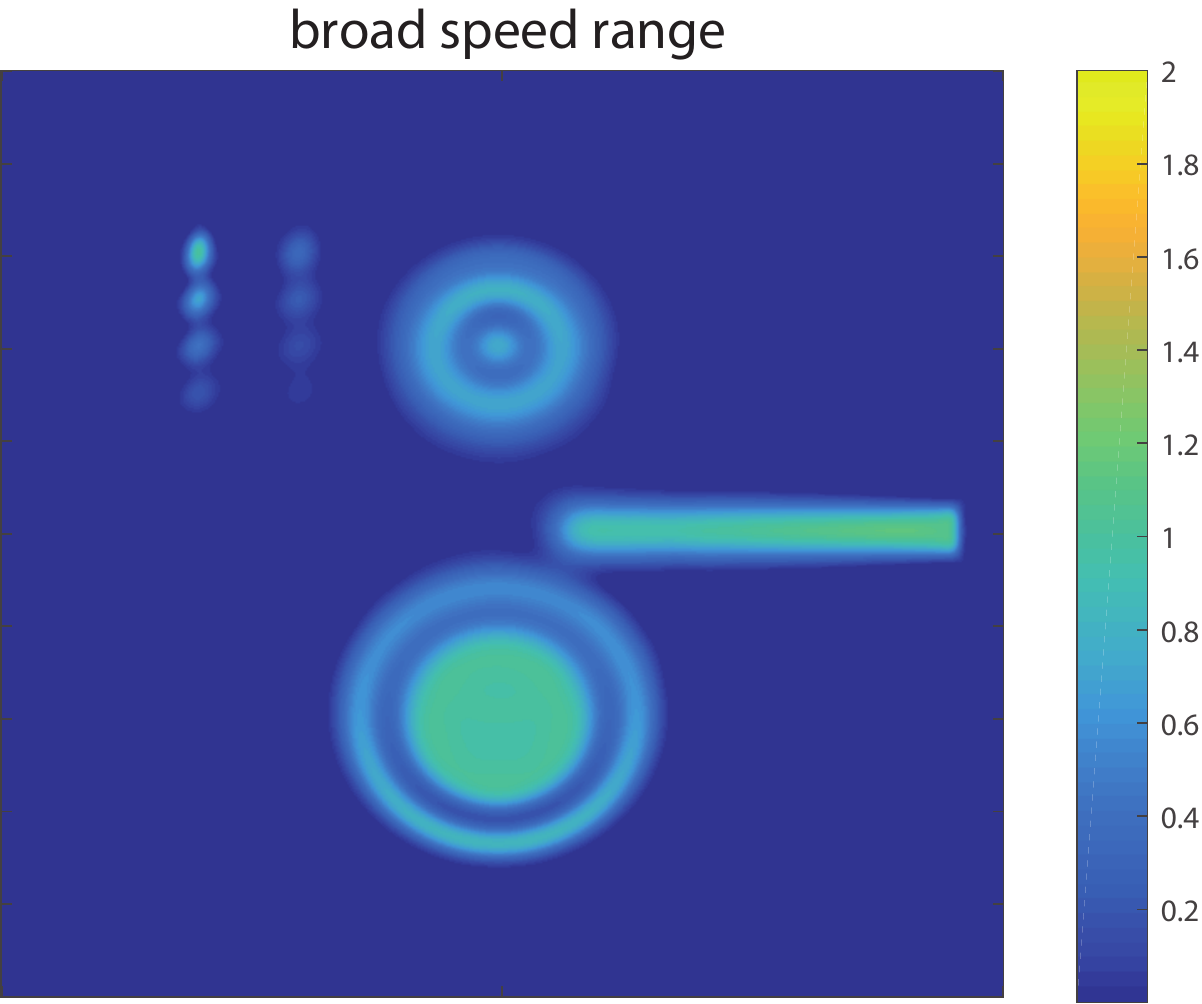}
\caption{\textsc{Effects of increasing the speed range $[c_0, c_\infty]$.}
Left: Noisy un-attenuated pressure and attenuated pressure  for
$c_\infty = \unit[1623]{m/s}$ and $c_\infty =\unit[3080]{m/s} $, respectively.
Right: Reconstruction for  $c_0 = \unit[1540]{m/s}$ and $c_\infty = \unit[3080]{m/s}$. (The reconstruction for
$c_0 = \unit[1540]{m/s}$ and $c_\infty = \unit[1623]{m/s}$ is shown in  the top
right image in  Figure~\ref{fig:strong}.)}
\label{fig:largecinf}
\end{figure}

We  are not aware how to exactly choose the free parameters
$c_0$ and $c_\infty$  in order to accurately model acoustic attenuation in water or soft tissue.
To investigate the effect of changing these parameters, we also perform simulations with a significantly increased  value of $c_\infty = \unit[3080]{m/s} $. From the results showing  in Figure~\ref{fig:largecinf}, one observes significantly increased attenuation compared to the value
$c_\infty = \unit[1623]{m/s}$ (see top right image in  Figure~\ref{fig:strong}).

\begin{figure}[htb!]
\includegraphics[width=0.495\textwidth]{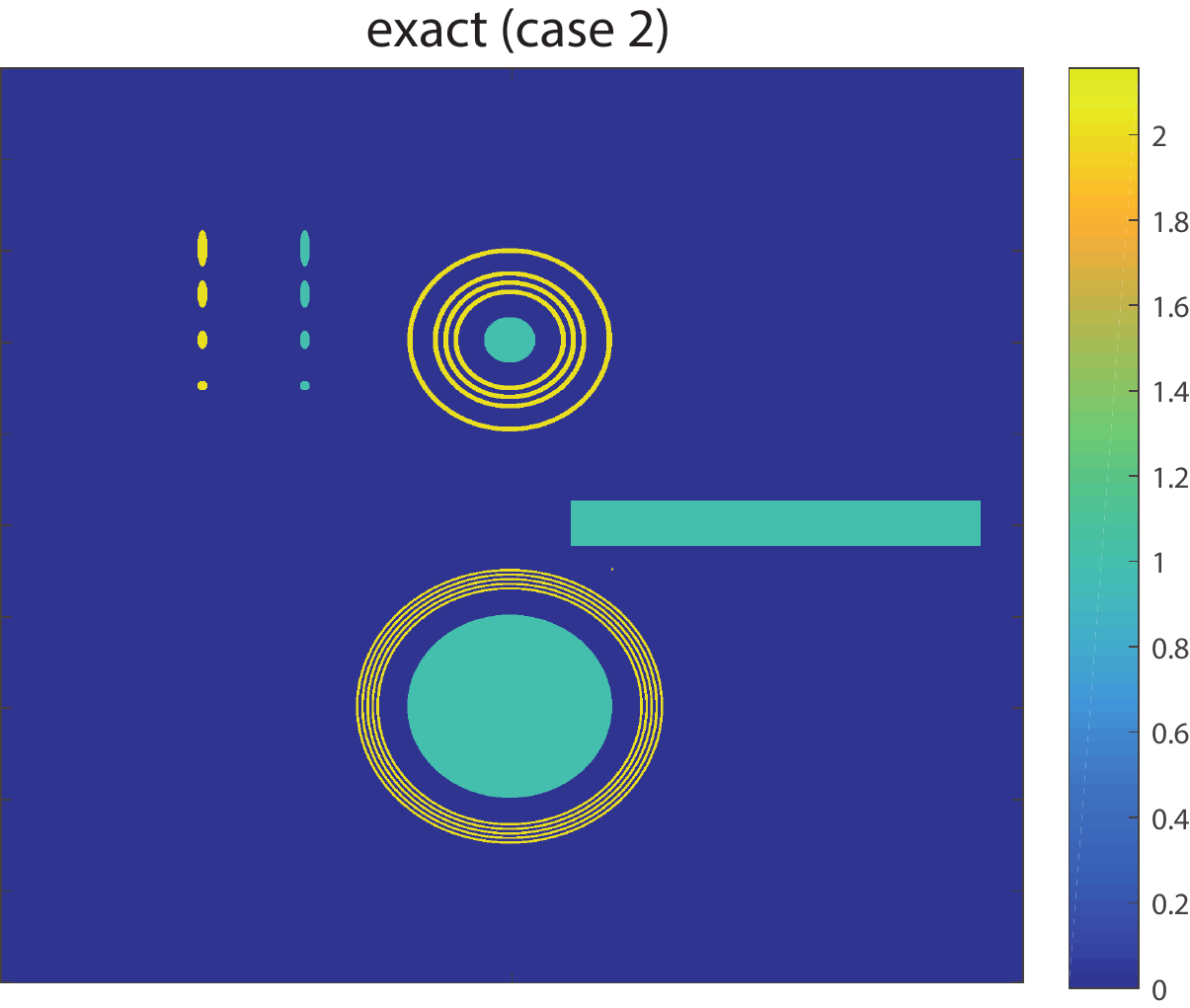}
\includegraphics[width=0.495\textwidth]{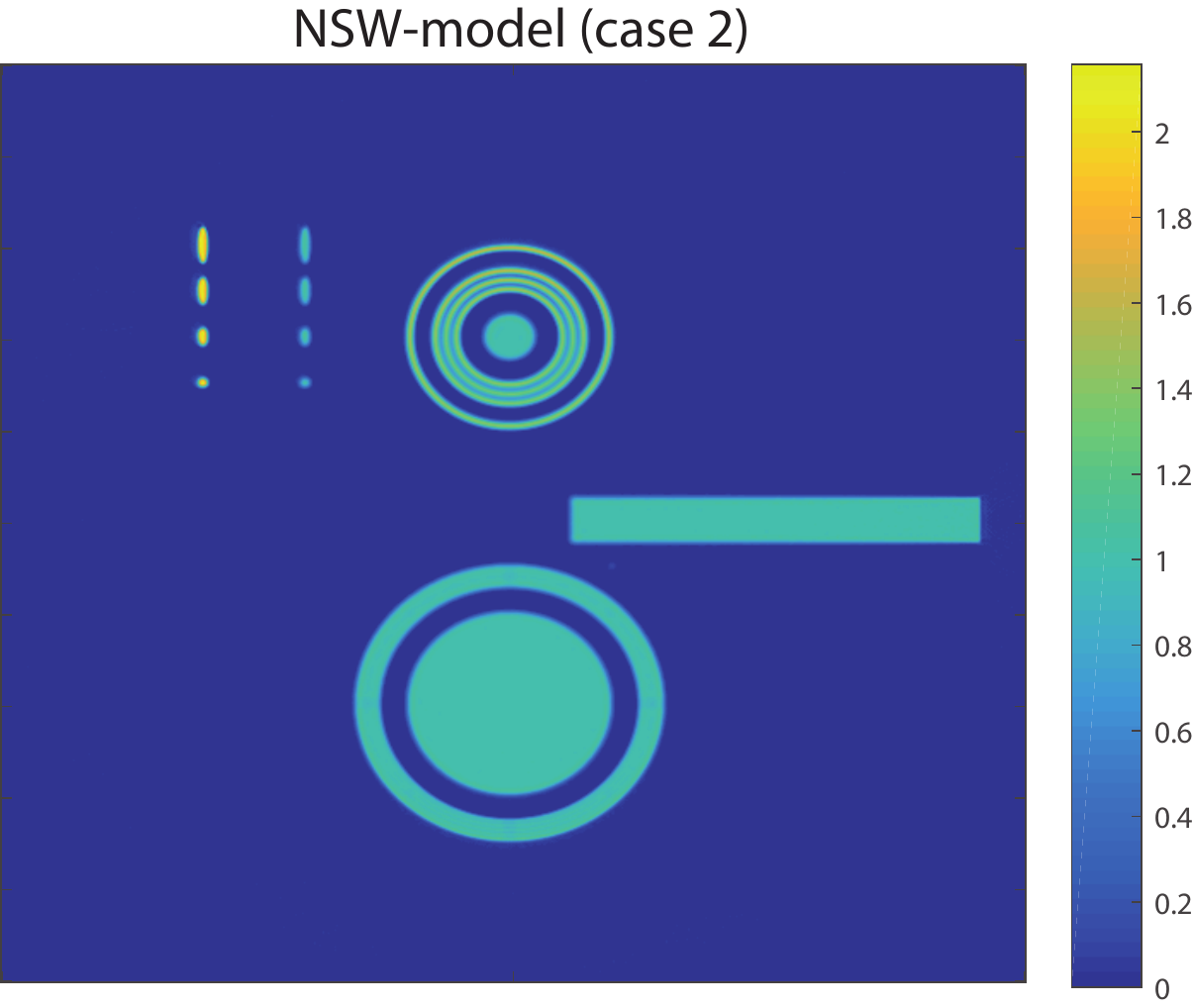}\\
\includegraphics[width=0.495\textwidth]{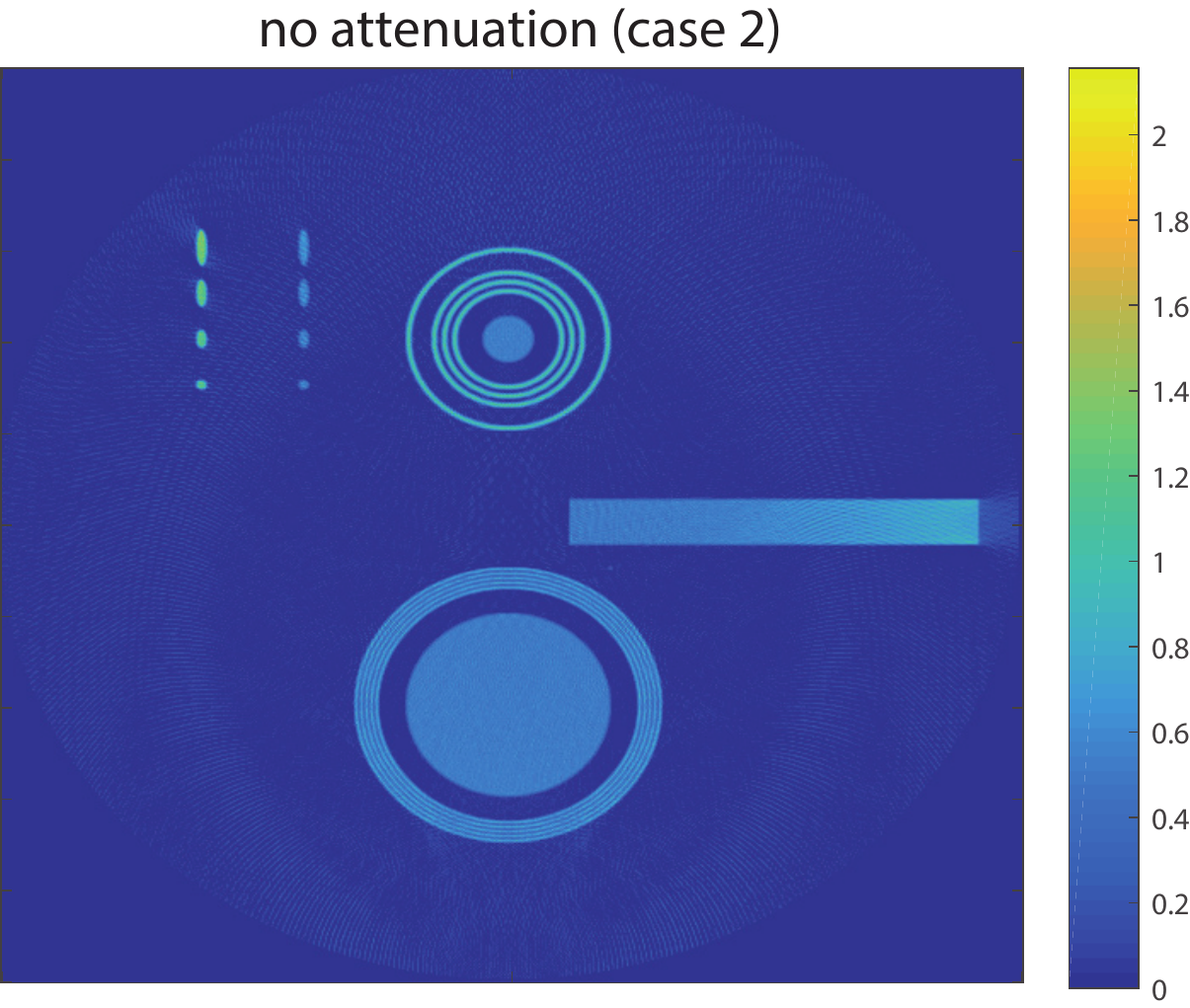}
\includegraphics[width=0.495\textwidth]{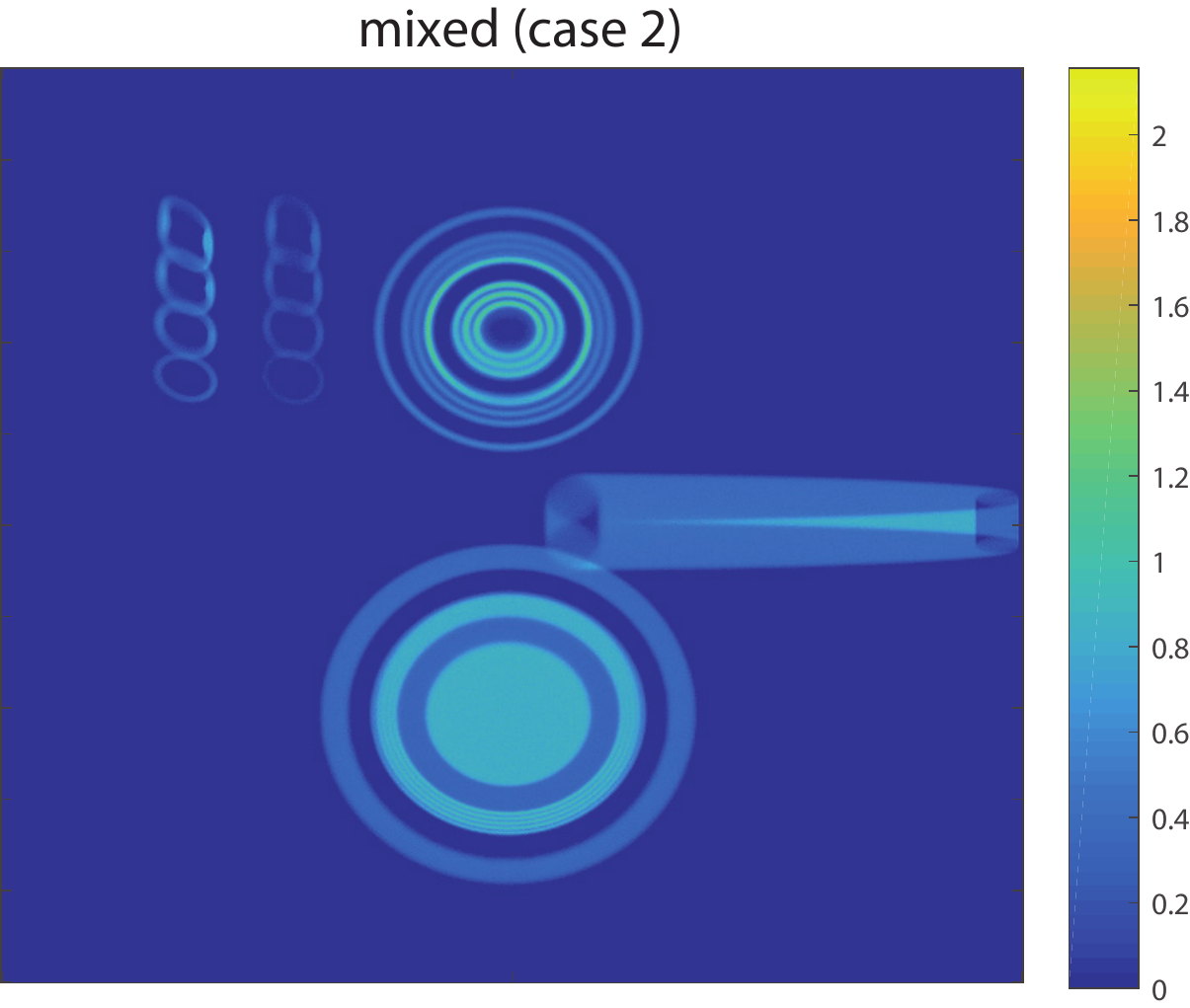}
\caption{\textsc{Reconstructions in the weak attenuation case $\tau_1 = \unit[1]{ns}$.}
Top left:  Exact PA source $\hnum$. Top right: Reconstruction based on the NSW model. Bottom left:  Reconstruction  in the absence of attenuation.
Bottom right: Reconstruction from attenuated data  but neglecting
attenuation in the reconstruction.}
\label{fig:weak}
\end{figure}

\subsection{Reconstruction results for weak attenuation}

Now we present simulations of the NSW model for weak attenuation
case with  relaxation time $\tau_1 = \unit[1]{ns} $.
As a consequence, we have to use a finer time discretization for
calculating   $\mnum_\alpha$ in \eqref{eq:Mn2}, \eqref{eq:Mn3}.
In order to keep the computational expenses reasonable,
we decreased the radius to $\RR =  \unit[5]{mm}$.
From the numerical results presented in Figure~\ref{fig:weak}, we
see that the attenuated case again yields a smoother reconstructions
than in the absence of attenuation. In contrast to the strong attenuation case, the very thin concentric annuli located in
the upper half of the image of $f$ can still be resolved; the contrast now even seems  better than in the dissipation free case.
Also, the very small elliptic structures can be estimated with high quality.  The thinner concentric annuli located in the lower half  of the image of $f$ cannot be resolved. As in the case of strong attenuation, if the  standard wave reconstruction  is applied attenuated data, then the reconstruction of thin and long structures is blurred and displaced.

\begin{figure}[htb!]
\includegraphics[width=0.495\textwidth]{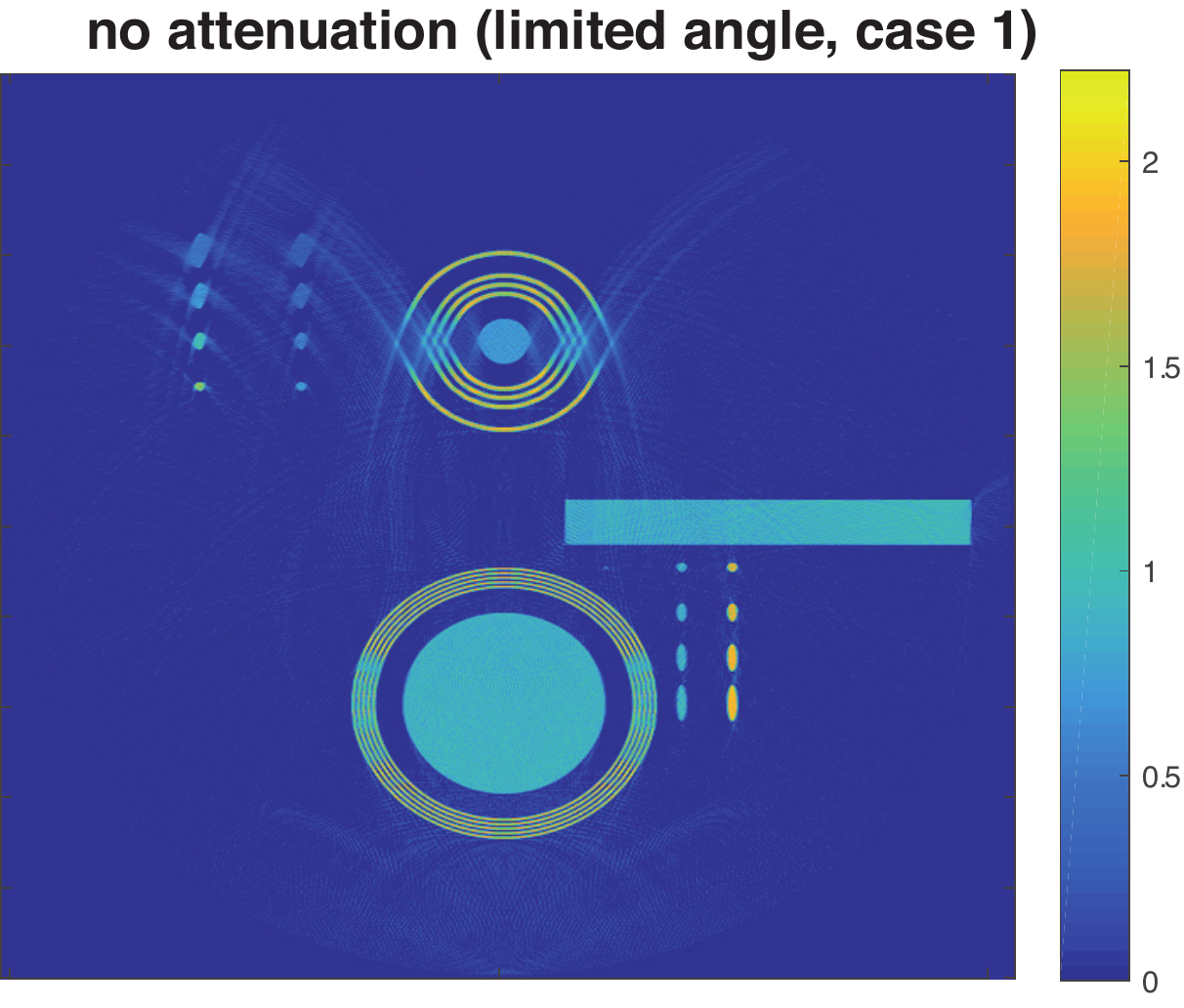}
\includegraphics[width=0.495\textwidth]{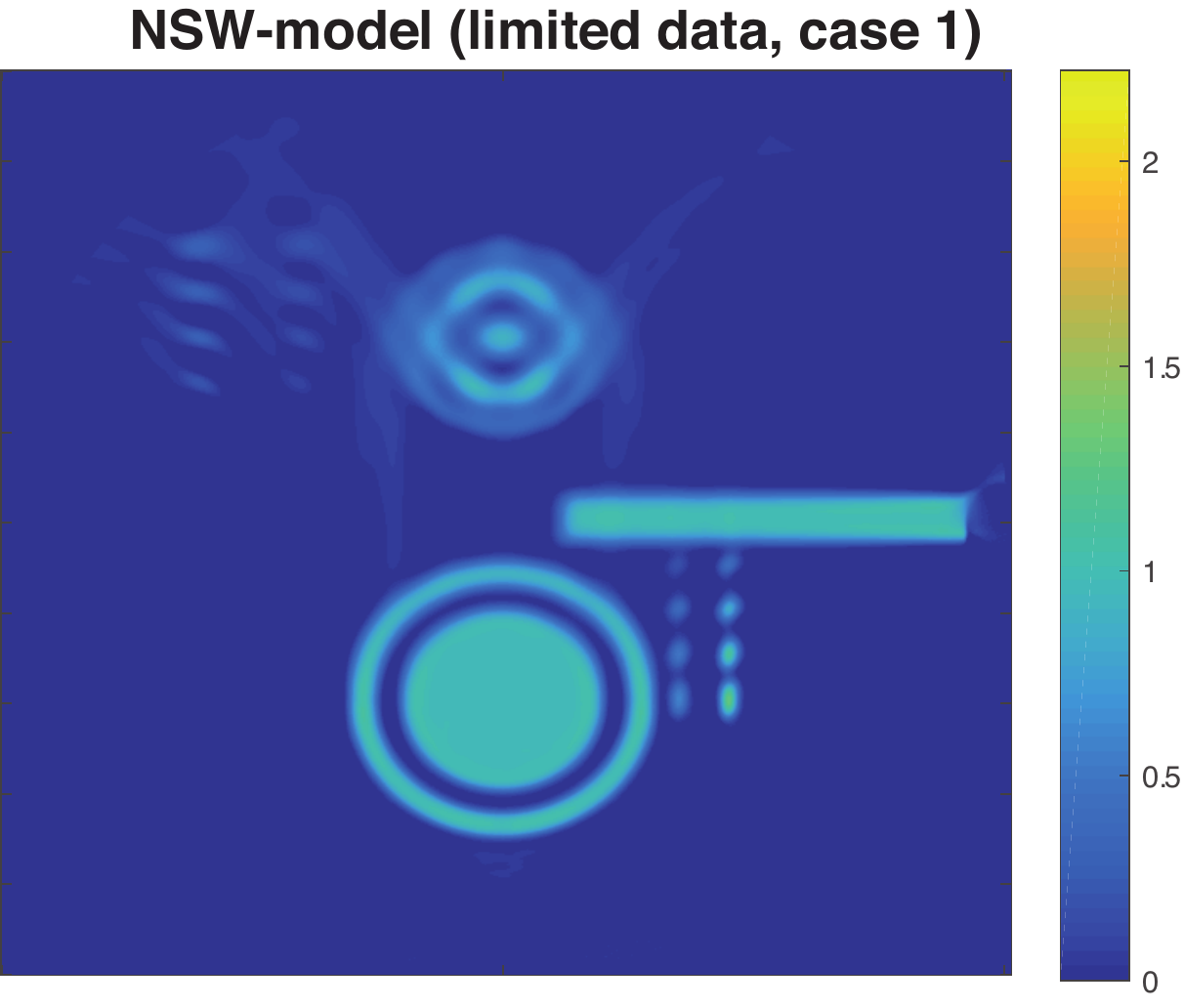}\\
\includegraphics[width=0.495\textwidth]{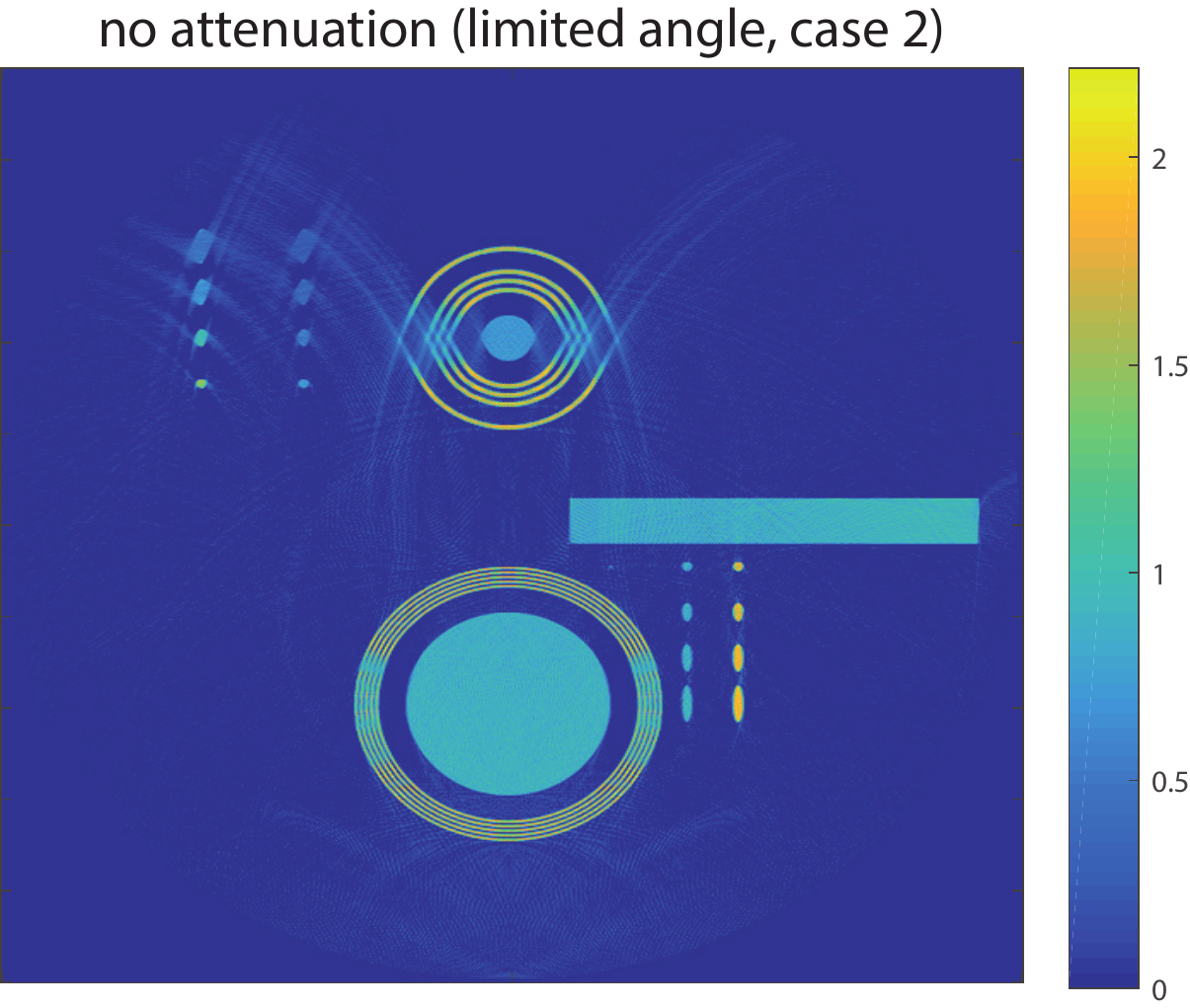}
\includegraphics[width=0.495\textwidth]{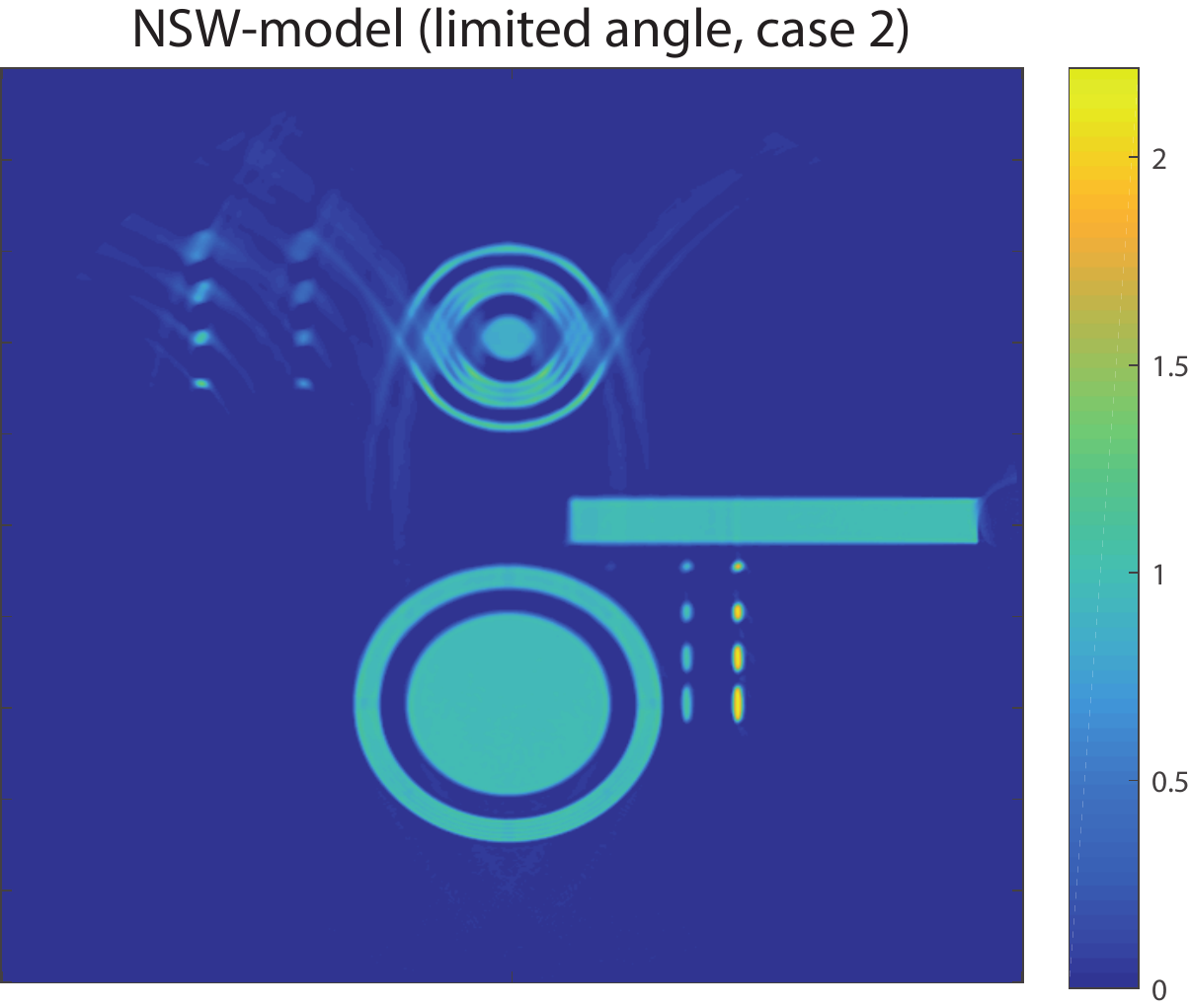}
\caption{\textsc{Reconstructions from limited view data}. Top row shows the reconstruction in the strong attenuation case using attenuation
free data (left)  attenuated data (right).
The bottom row shows the same for the  weak  attenuation case.}
\label{fig:limited}
\end{figure}

\subsection{Reconstruction results for limited view data}

Finally, we perform reconstructions using limited view (or limited angle) data, where the detector positions $\x= \RR(\cos \ph, \sin \ph)$ are located on a half circle with  $\ph \in [0,\pi]$. The reconstruction results using the projected Landweber method are shown in Figure~\ref{fig:limited}. The top row considers the  the strong attenuation  case and the bottom row the weak attenuation case. In both cases, the result are compared to the cases without attenuation. All reconstructions show the typical limited view artefacts, even in the absence of attenuation. Further, one notices that the  results using un-attenuated   data yield a little better contrast for long thin structure and much better contrast for structures with small diameters than the ones with attenuated data.
Again, we see if the attenuation is not too strong, then attenuation leads to  smoother images
and even partly better results  than in the absence of attenuation.

\section{Conclusion}
\label{sec:conclusion}

In this paper, we developed iterative regularization methods for PAT in attenuating media.
This comes with a clear   convergence theory in the Hilbert space framework that is not shared by
any other existing approach. For the sake of clarity, we focused on the  Landweber method.
Generalizations to other regularization techniques such the CG method or Tikhonov regularization
are subject of future research.    A main ingredient of these regularization methods is the
evaluation adjoint of the forward operator. For that purpose, we developed two formulations for the adjoint:
One takes the form of an explicit formula whereas the second one involves the solution  of an adjoining wave equation.
While the proposed method can equally
be applied  for general admissible attenuation models, in or numerical numerical results, we focused on the widely
accepted attenuation model  of Nachman, Smith and Waag \cite{nachman1990equation}.  A detailed comparison
of  reconstructions with different attenuation laws  and different  reconstruction algorithms is intended for future research.
The presented numerical results  clearly demonstrate that for moderate attenuation even small structures are estimated well
with our method.   On the other hand, our results show that not  accounting for attenuation yields severe
artifacts due to dispersion.  This clearly demonstrated the necessity of   taking correct attenuation models
into account in the inversion process. Moreover, our numerical experiments indicate that  for weak attenuation
the results are even better than in the absence of attenuation.

\section*{Acknowledgement}
L.V. Nguyen's research is partially supported by the NSF grants DMS 1212125 and DMS 1616904.
He also thanks the University of Innsbruck  for financial support and hospitality during his visit.

\appendix

\section{Adjoint attenuated wave equation}
\label{ap:proof}

Let $\Omout$ be an open set such that $\partial\Omout$
is a closed smooth surface in $\R^d$. For notational convenience, we will denote $\Omout_c \coloneqq  \R^d \setminus \overline \Omout$.
For $g \in C_0^\infty(\partial \Omout  \times \R) $ consider the  equation
\begin{linenomath} \begin{equation}  \label{eq:adjoint} \left\{
\begin{aligned}
&\kl{ \Do_\al  +
\frac{1}{c_0}\frac{\partial}{\partial t }  }^2  u(x,t)
-
\Delta u(x,t)
=  - \delta_{\partial \Omout}(x) \, g(x,t)
&&\text{ on } \R^d \times \R \,, \\
& \quad u (x,t) = 0  &&  \text{ for } t \ll 0   \,.
\end{aligned}
\right. \end{equation}\end{linenomath}
Here the notation $u (\edot ,t) = 0$  for  $t \ll 0 $ means that there exists some $t_0 \in (-\infty, 0]$
such that $u (\edot ,t) = 0$ for all $t  < t_0$. In this appendix, we show regularity of
solutions of \eqref{eq:adjoint} and demonstrate that its solution defines the adjoint of $\Wo_\al$.

\subsection{Regularity and  classical solution of~\eqref{eq:adjoint}} \label{app:existence}

For any $g \in C_0^\infty(\partial \Omout  \times \R) $,  the source  term  $- \delta_{\partial\Omout}(x) \, g(x,t)$ is a tempered distribution that vanishes for sufficiently small $t$. Hence~\eqref{eq:adjoint}
has a unique distributional solution
\begin{linenomath} \begin{equation}  \label{eq:single-space}
	u(x,t) = \int_\R \int_{\partial\Omout} G_\al(x-y,t-\tau)  \, g(y,\tau) \, \ds(y) \, d\tau \,,
\end{equation}\end{linenomath}
where  $G_\al  \in \sr'(\R^d \times \R)$ is the causal Greens function of  the attenuated wave
equation~\eqref{eq:w}.
As first  result in this section we show  that  the restrictions of the
 solution  to $\Omout$ and $\Omout_c$  are smooth and both
 can be smoothly extended  to $\partial \Omout$.

\begin{theorem}[Regularity of solutions of \eqref{eq:adjoint}]
For \label{thm:regularity} any $g \in C_0^\infty(\partial \Omout  \times \R) $,
\eqref{eq:adjoint} has a unique  solution $u \in C^\infty((\R^d \setminus \partial \Omout )\times \R) $. Further, $u$ can be  extended
continuously to $\R^d \times \R$, and $\rest{\nabla u}{\Omout}$ and $\rest{\nabla u}{\Omout_c}$ can be extended continuously to $\overline \Omout\times \R$ and $\overline{ \Omout_c} \times \R$, respectively.
 \end{theorem}

\begin{proof}
Let $u \in \sr'(\R^d \times \R)$ denote the unique distributional solution of~\eqref{eq:w}
with $s(x, t)  = \delta_{\partial\Omout} (x) g(x,t)$, given by \ref{eq:single-space}.
In order to obtain the regularity of $u$ we work in the frequency domain
and employ the theory of single and double layer potentials
for the Helmholtz equation.
For that purpose note that the temporal Fourier transform of
$G_\al$ is given by $\Phi_\al(x,\om) = {e^{\imi k(\om) \abs{\x}}}/{\abs{\x}}$ where $k(\om)  \coloneqq \imi \alpha(\om) + \om/c_0$.
Further, write $\hat u$ and $\hat g$ for the temporal Fourier transform of
$u$ and $g$, respectively. Then
\begin{linenomath} \begin{equation}  \label{eq:single-fouier}
	\forall (x,\om) \in (\R^d\setminus \partial \Omout) \times \R \colon
	\quad
	\hat u (x,\om)  = \int_{\partial\Omout} \Phi_\al(x-y,\om) \, \hat g(y,\om) \, \ds(y) \,,
\end{equation}\end{linenomath}
which is recognized as a single layer potential for the Helmholtz equation with density
$\hat g$. Since $\hat g (\edot , \om) \in C^1(\partial \Omout)$, the
theory of single and double layer potentials (see, for example, \cite{colton2013integral}) shows the following:
\begin{itemize}
\item  $\rest{[\hat u (\edot,\om)]}{\partial\Omout} =0$ and $\rest{[\partial_\nu  \hat u(\edot,\om)]}{\partial\Omout}=\hat g(\edot,\om)$;

\item  $\nabla \hat u(\edot,\om) \in C(\overline \Omout)$ and $\nabla \hat u(\edot,\om) \in C(\overline{\Omout_c})$;

\item
For some function $C(\om)$  that is at most polynomially growing, we have
\begin{linenomath} \begin{equation} \label{eq:helm-aux}
\norm{\rest{\hat u (\edot,\om)}{(\R^d \setminus \partial\Omout)}}_{\infty}+
\norm{\rest{\nabla \hat u(\edot,\om)}{(\R^d \setminus \partial\Omout)}}_\infty
 \leq
C(\om) \snorm{\hat g(\edot,\om)}_{C^1} \,.
\end{equation}\end{linenomath}
\end{itemize}
Here and below the bracket $[v]$ denotes the jump of a function
$v \in  C(\overline{\Omout} \times \R) \cap C(\overline{\Omout_c} \times \R)$
 across the surface $\partial\Omout$ (from inside out).
Next note that $\om \mapsto \snorm{\hat g(\edot,\om)}_{C^1}$ decays faster than any polynomial. Therefore, \eqref{eq:helm-aux} implies that $u$ is infinitely differentiable on $C(\R^d)$ and  that $\rest{\nabla u}{\Omout}$ and $\rest{\nabla u}{\Omout_c}$ are infinitely differentiable on $C(\Omout)$ and $C(\Omout_c)$ (with respect to the time variable). \end{proof}

In the following, we call $u \in C^\infty((\R^d \setminus \partial \Omout )\times \R) $ a classical solution of \eqref{eq:adjoint},  if
$u$ can be  extended  continuously to $\R^d \times \R$,
$\rest{\nabla u}{\Omout}$ and $\rest{\nabla u}{\Omout_c}$ can be continuously extended to $\partial \Omout$, and
\begin{linenomath} \begin{equation}  \label{eq:classical}\left\{
\begin{aligned}
&\kl{ \Do_\al  +
\frac{1}{c_0}\frac{\partial}{\partial t }  }^2  u(x,t)
-
\Delta u(x,t)
= 0
&& \text{ for } (x,t) \in (\R^d \setminus \partial\Omout) \times \R \,,
\\
& \quad [\pd_\nu u](x,t) =  g(x,t),
&&  \text{ for } (x,t) \in \partial\Omout \times \R \,,
 \\
& \quad u (\edot ,t) = 0
&&  \text{ for } t \ll 0  \,.
\end{aligned}
\right. \end{equation}\end{linenomath}
Using Theorem~\ref{thm:regularity}, we one can show the
following existence and uniqueness result  for solutions
of~\eqref{eq:classical}.

\begin{corollary}[Existence and uniqueness of   \eqref{eq:classical}]
Any classical solution of \eqref{eq:classical} is a distributional
solution of \eqref{eq:adjoint}, and vice versa.
 In particular, for any $g \in C_0^\infty (\R \times \partial \Omout)$,
 \eqref{eq:classical} is uniquely solvable.
\end{corollary}

\begin{proof}
According to Theorem~\ref{thm:regularity} and its
proof, any solution of~\eqref{eq:adjoint}  is a solution of
\eqref{eq:classical}. Conversely,  note that  $u$ is a distributional solution of~\eqref{eq:adjoint} if and only
\begin{linenomath} \begin{equation}  \label{eq:ad1}
	\forall
	\phi \in C_0^\infty(\R^d \times \R)
\colon
\quad
\inner{u}{\phi} = \int_{\R} \int_{\partial\Omout} g(x,t) \Phi (x,t) \, \ds(x) \, dt
	\,,
\end{equation}\end{linenomath}
where  $\Phi$ is the solution of  $\bigl( \Do_\al^*  -
\tfrac{1}{c_0} \tfrac{\partial}{\partial t }  \bigr)^2  \Phi
- \Delta \Phi
=  \phi$ and $\Phi (\edot ,t) = 0 $ for $t > T$.
Using integration by parts one  verifies that any
solution of \eqref{eq:classical}  satisfies
\eqref{eq:ad1} and therefore  also~\eqref{eq:adjoint}.
\end{proof}

\subsection{Proof of Theorem~\ref{thm:adjointw1}}

Let $h \in C^\infty_0(\Omout)$
and let $p_\al$ be the solution of the attenuated wave
equation~\eqref{eq:wavealpha}. Multiplication of \eqref{eq:wavealpha}
with a test function $\psi \in C_0^\infty(\R^d  \times \R) $ and integrating by parts shows
\begin{linenomath} \begin{multline} \label{E:Var}
\int_{\R} \int_{\R^d}  \ekl{ \kl{ \Do_\al  -
\frac{1}{c_0}\frac{\partial}{\partial t }  }^2  p_\al(x,t) }\, \psi(x,t) \,\rmd x \, \rmd t \\
- \int_{\R} \int_{\R^d} p_\al(x,t) \ekl{  \Delta \psi(x,t)}  \,\rmd x \, \rmd t  =
-
\int_{\R^d}  h(x) \, \frac{ \partial \psi}{\partial t}(x,0) \, \rmd x  \,. \end{multline}\end{linenomath}
Now suppose $g \in C_0^\infty(\Gamma \times (0,\tmax) )$ and let $q_\al \in C^\infty(\R^d \times \R)$ be the solution of the adjoint attenuated wave equation~\eqref{eq:attenuated1}. Because $p_\al$ vanishes for $t<0$ and $q_\al$ vanishes for $T>0$ ,  identity~\eqref{E:Var}
also holds for $\psi = q_\al$. This  gives
\begin{linenomath} \begin{multline} \label{E:Iden}
\int_{\R} \int_{\R^d}  \ekl{ \kl{ \Do_\al  -
\frac{1}{c_0}\frac{\partial}{\partial t }  }^2  p_\al(x,t) }\, q_\al(x,t) \,\rmd x \, \rmd t  \\
-  \int_{\R} \int_{\R^d} p_\al(x,t) \ekl{  \Delta q_\al (x,t)}  \,\rmd x \, \rmd t
= -
\int_{\R^d}  h(x) \, \frac{ \partial q_\al}{\partial t}(x,0) \, \rmd x  \,.
\end{multline}
\end{linenomath}
Using  two times integration by parts in the  first term in \eqref{E:Iden}
with respect to $t$, using the definition of $\Do_\al^*$,
and recalling that $q_\al$ solves the adjoint attenuated wave equation
\eqref{eq:attenuated1}   shows
\begin{linenomath}
\begin{multline*}
-\int_{\R^d}  h(x) \, \frac{ \partial q_\al}{\partial t}(x,0) \, \rmd x
=
 \int_{\R} \int_{\R^d}  p_\al(x,t)  \, \ekl{ \kl{ \Do_\al^*  +
\frac{1}{c_0}\frac{\partial}{\partial t }  }^2 q_\al(x,t) - \Delta q_\al(x,t)   }\,\rmd x \, \rmd t \\
= - \int_{\R} \int_{\Gamma}  p_\al(y,t)  g_\al(y,t)\,\ds (y) \, \rmd t \,.\end{multline*}
\end{linenomath}
Consequently, for every $g \in C_0^\infty(\Gamma \times (0, \infty))$, it holds that
\begin{linenomath}
\begin{equation*}
\forall h\in C_0^\infty(\Omout) \colon \quad
 \int_0^\tmax \int_{\Gamma}  (\Wo_\al h)(y,t)  \,  g(y,t) \, \ds(y)
\, \rmd t
=
 \int_{\R^d}  h(x) \, \frac{ \partial q_\al}{\partial t}(x,0) \, \rmd x
 \,.
\end{equation*}\end{linenomath}
As the last identity holds on a  dense subset of
$L^2(\Omout)$, this shows the expression \eqref{eq:attenuated1}, \eqref{eq:adjointw1} for $\Wo_\al^*  g$ and concludes
the  proof of Theorem~\ref{thm:adjointw1}.

\end{document}